\documentclass[headings=small]{scrartcl}
\pdfoutput=1

\newcommand{\textwidthScaling}{0.98}

\input{Header}

\newcommand{\idx}{j}
\newcommand{\idxx}{k}

\def\tau{{t'}}

\let\defaultCompany\empty
\DeclareRobustCommand{\nonDefaultCompany}[1]{%
\ifthenelse{\equal{#1}{\defaultCompany}}{}{%
		#1%
	}%
}
\DeclareRobustCommand{\nonDefaultCompanyKern}[2]{%
\ifthenelse{\equal{#1}{\defaultCompany}}{}{%
		#2\relax#1%
	}%
}
\newcommand{\nonEmptyKern}[2]{%
\ifthenelse{\isempty{#1}}{}{%
		#2\relax#1%
	}%
}

\newcommand{\Units}{\mkern-2mu\mathcal{I}\mkern-2mu}
\newcommand{\Periods}{[T]}

\newvar{periodLength}{1}{\textup{L}^{#1}}
\newvar{minProd}{1}{\underline{\textup{P}}_{\nonDefaultCompany{#1}}}
\newvar{maxProd}{1}{\overline{\textup{P}}_{\nonDefaultCompany{#1}}}
\newvar{heatingProd}{1}{\overline{\textup{HP}}_{\nonDefaultCompany{#1}}}
\newvar{heatingSpeed}{1}{\overline{\textup{H}}_{\nonDefaultCompany{#1}}}
\newvar{maxHeating}{2}{\overline{\textup{h}}_{\nonDefaultCompany{#1}}^{#2}}
\newvar{minUptime}{1}{\textup{UT}_{\nonDefaultCompany{#1}}}
\newvar{minDowntime}{1}{\textup{DT}_{\nonDefaultCompany{#1}}}
\newvar{initialOffline}{1}{\textup{IDT}_{\nonDefaultCompany{#1}}}
\newvar{initialOnline}{1}{\textup{IUT}_{\nonDefaultCompany{#1}}}
\newvar{rampup}{1}{\textup{RU}_{\nonDefaultCompany{#1}}}
\newvar{rampdown}{1}{\textup{RD}_{\nonDefaultCompany{#1}}}
\newvar{startupRamp}{1}{\textup{SU}_{\nonDefaultCompany{#1}}}
\newvar{shutdownRamp}{1}{\textup{SD}_{\nonDefaultCompany{#1}}}
\newvar{preOffline}{1}{\textup{PDT}_{\nonDefaultCompany{#1}}}

\newvar{StartupCostFunction}{2}{%
\textup{CU}_{\nonDefaultCompany{#1}}%
\ifthenelse{\isempty{#2}}{}{({#2})}%
}
\newvar{StartupCostFunctionTilde}{2}{%
\tilde{\textup{CU}}_{\nonDefaultCompany{#1}}%
\ifthenelse{\isempty{#2}}{}{\left({#2}\right)}%
}
\newvar{StartupCost}{3}{\textup{CU}_{\nonDefaultCompany{#1}}^{#2}\ifthenelse{\isempty{#3}}{}{({#3})}}
\newvar{StartupCostTilde}{3}{\tilde{\textup{CU}}_{\nonDefaultCompany{#1}}^{#2}\ifthenelse{\isempty{#3}}{}{({#3})}}
\newvar{ShutdownCost}{1}{\textup{CD}_{\nonDefaultCompany{#1}}}
\newvar{heatloss}{1}{\lambda_{#1}}
\newvar{heatingCost}{1}{V_{#1}}
\newvar{fixedStartupCost}{1}{F_{#1}}

\newvar{fuelType}{1}{\textup{F}_{\nonDefaultCompany{#1}}}
\newvar{fuelCost}{2}{\textup{FC}_{\nonDefaultCompany{#1}}^{#2}}
\newvar{varFuel}{1}{\textup{FA}_{#1}}
\newvar{fixedFuel}{1}{\textup{FB}_{#1}}
\newvar{varCost}{1}{\textup{PA}_{#1}}
\newvar{fixedCost}{1}{\textup{PB}_{#1}}

\newvar{onOff}{2}{v^{\nonEmptyKern{#2}{\mkern2mu}}_{\nonDefaultCompanyKern{#1}{\mkern-2mu}}}
\newvar{tildeOnOff}{2}{\tilde{v}^{\nonEmptyKern{#2}{\mkern2mu}}_{\nonDefaultCompanyKern{#1}{\mkern-2mu}}}
\renewvar{prod}{2}{p_{\nonDefaultCompany{#1}}^{#2}}
\newvar{startupCost}{2}{\textup{cu}_{\nonDefaultCompany{#1}}^{#2}}
\newvar{tildeStartupCost}{2}{\tilde{\textup{cu}}_{\nonDefaultCompany{#1}}^{#2}}
\newvar{startupIndicator}{2}{\textup{sui}_{\nonDefaultCompany{#1}}^{#2}}
\newvar{tildeStartupIndicator}{2}{\tilde{\textup{sui}}_{\nonDefaultCompany{#1}}^{#2}}
\newvar{shutdownIndicator}{2}{\textup{sdi}_{\nonDefaultCompany{#1}}^{#2}}
\newvar{shutdownCost}{2}{\textup{cd}_{\nonDefaultCompany{#1}}^{#2}}
\newvar{startupType}{3}{\delta_{#1}^{#2}(#3)}
\newvar{temp}{2}{\textup{temp}_{#1}^{#2}}
\newvar{heating}{2}{h_{#1}^{#2}}\newvar{prodCost}{2}{\textup{cp}_{\nonDefaultCompany{#1}}^{#2}}

\newvar{startupIndicatorFunc}{1}{\textup{sui}_{\onOff(i,)}^{#1}}
\newvar{tildeStartupIndicatorFunc}{1}{\tilde{\textup{sui}}_{\tildeOnOff(i,)}^{#1}}
\newvar{shutdownIndicatorFunc}{1}{\textup{sdi}_{\onOff(i,)}^{#1}}
\newvar{tildeShutdownIndicatorFunc}{1}{\tilde{\textup{sdi}}_{\tildeOnOff(i,)}^{#1}}
\newvar{startupCostFunc}{1}{\textup{cu}_{\onOff(i,)}^{#1}}
\newvar{tildeStartupCostFunc}{1}{\tilde{\textup{cu}}_{\tildeOnOff(i,)}^{#1}}

\newvar{offlineNumber}{2}{\textup{ol}^{#1}\ifthenelse{\isempty{#2}}{}{(#2)}}
\newvar{offlineNumberRight}{2}{\textup{or}^{\mkern2mu #1}(#2)}
\newvar{offlineFromTo}{2}{\textup{OL}(#1\textvisiblespace[1.5ex]#2)}
\newvar{offlineLength}{2}{\textup{OL}^{#1}\ifthenelse{\isempty{#2}}{}{(#2)}}
\newvar{offlineLengthTree}{1}{\textup{OL}(#1)}
\renewvar{offlineFromTo}{2}{\textup{L}_{#1,#2}}

\newvar{DiscreteStartupCost}{2}{\textup{DCU}_{\nonDefaultCompany{#1}}^{#2}}

\newcommand{\DiscreteStartupCostSum}{\textup{DCU}_{\nonDefaultCompany{i}}^\Sigma}
\newcommand{\StartupCostSum}{\textup{LCU}_{\nonDefaultCompany{i}}^\Sigma}
\newcommand{\startupCostSum}{\startupCost(i,\Sigma)}

\newcommand{\scsEpigraph}{\epi(\StartupCostSum)}
\newcommand{\scsVertices}{{V^{\Sigma}}}
\newcommand{\scsBinaryTrees}{{\mathcal{B}}}
\newcommand{\scsSumDiff}[1]{\delta^{#1}}
\newvar{bticoeff}{1}{\alpha_{\mkern2mu #1}}

\newcommand{\newTreeFunction}[2]{
	\newvar{#1T}{2}{{#2}^{##1}\ifthenelse{\isempty{##2}}{}{(##2)}}
	\newvar{#1}{1}{{#2}\ifthenelse{\isempty{##1}}{}{(##1)}}
}

\newTreeFunction{depth}{d}
\newTreeFunction{rank}{\textup{rank}}
\newvar{rankInv}{1}{\textup{rank}^{-1}\ifthenelse{\isempty{#1}}{}{(#1)}}

\newTreeFunction{subtree}{S}
\newTreeFunction{size}{s}
\newTreeFunction{leftSubtree}{L}
\newTreeFunction{leftSize}{l}
\newTreeFunction{rightSubtree}{R}
\newTreeFunction{rightSize}{r}
\newcommand{\oneSubtree}[2][B]{#1({#2})}

\let\root\relax
\newvar{root}{1}{\textup{root}\ifthenelse{\isempty{#1}}{}{(#1)}}
\newTreeFunction{parent}{p}
\newTreeFunction{leftChild}{\textup{llink}}
\newTreeFunction{rightChild}{\textup{rlink}}

\newvar{topleftT}{2}{\mu^{#1}_{#2}}
\newvar{topleft}{1}{\topleftT{}{#1}}
\newcommand{\topleftN}{\depth(1)+1}
\newvar{toprightT}{2}{\eta^{#1}_{#2}}
\newvar{topright}{1}{\toprightT{}{#1}}
\newcommand{\toprightN}{{\depth(T)+1}}

\colorlet{onOffColor}{green}
\colorlet{onOffText}{green!50!black}
\colorlet{onOffBack}{green!20}

\colorlet{prodColor}{blue}
\colorlet{prodText}{blue!80!black}

\colorlet{tempColor}{red}
\colorlet{tempText}{red!80!black}

\colorlet{fuelNeedColor}{red}
\colorlet{efficiencyColor}{blue}

\colorlet{cooling}{blue!70!green}
\colorlet{coolingBack}{cooling!8!white}
\colorlet{coolingText}{cooling!80!black}
\colorlet{heating}{red}
\colorlet{heatingBack}{red!8!white}
\colorlet{heatingText}{red!70!black}

\newcommand{\forallPeriods}{\forall\, t \in \Periods}
\newcommand{\forallUnitsPeriods}{\forall\, i \in \Units, t \in \Periods}
\newcommand{\forallUnits}{\forall\, i \in \Units}

\makeatletter
\newcommand{\@rr}{ep\-re\-sen\-ta\-tion}
\newcommand{\@mc}[1]{\texorpdfstring{$\mathcal{#1}$}{#1}}
\newcommand{\hrepresentation}{\@mc{H}-r\@rr\xspace}
\newcommand{\Hrepresentation}{\@mc{H}-R\@rr\xspace}
\newcommand{\hrepresentations}{\@mc{H}-r\@rr s\xspace}
\newcommand{\Hrepresentations}{\@mc{H}-R\@rr s\xspace}
\newcommand{\vrepresentation}{\@mc{V}-r\@rr\xspace}
\newcommand{\Vrepresentation}{\@mc{V}-R\@rr\xspace}
\newcommand{\vrepresentations}{\@mc{V}-r\@rr s\xspace}
\newcommand{\Vrepresentations}{\@mc{V}-R\@rr s\xspace}
\makeatother

\usepackage{multicol}

\usepackage[a4paper,text={36.526315789em,230mm},heightrounded]{geometry}

\begin{document}

\author{Ren\'e Brandenberg\thanks{brandenb@ma.tum.de, Department of Mathematics} \and Matthias Huber\thanks{matthias.huber@tum.de, Department of Electrical and Computer Engineering} \and Matthias Silbernagl\thanks{silbernagl@tum.de, Department of Mathematics \newline The authors are with the Technische Universit\"at M\"unchen, 80333 M\"unchen, Germany.}}

\title{The Summed Start-up Costs in a Unit Commitment Problem}

\date{Preprint, 4.3.2015}

\maketitle

\begin{abstract}
	We consider the sum of the incurred start-up costs of a single unit in a Unit Commitment problem. Our major result is a correspondence between the facets of its epigraph and some binary trees for concave start-up cost functions~$\StartupCostFunction(,)$, which is bijective if $\StartupCostFunction(,)$ is strictly concave. We derive an exponential \hrepresentation of this epigraph, and provide an exact linear separation algorithm. These results significantly reduce the integrality gap of the Mixed Integer formulation of a Unit Commitment Problem compared to current literature.
\end{abstract}

\section{Introduction}
\label{section:SCS_IntroductionGeneral}

Start-up costs have been an integral part of the Unit Commitment problem since its first formulation \cite{garver_power_1962}, and gain further importance due to the increasing penetration of variable renewable energy. Electricity production from such sources, i.e. mainly from wind and solar, increases the variability of the net load met by thermal generators, resulting in a more frequent cycling and a higher number of start-ups (\cite{kumar_power_2012}). Hence, the percentage of start-up costs within the total production costs rises and an accurate modeling of the start-up costs becomes increasingly important.

In contrast to the detailed incorporation of start-up costs in other methodologies (\eg Lagrangean Relaxation, Dynamic Programming, meta-heuristics), Mixed Integer Programming formulations were restricted to constant start-up costs (\cite{garver_power_1962}) for a long time. The first model capable of representing arbitrary start-up cost functions was introduced in \cite{arroyo_optimal_2000}. Since then, two major formulations have emerged: the model of \cite{carrion_computationally_2006}, based on \cite{nowak_stochastic_2000}, and the model of \cite{simoglou_optimal_2010}. Their comparison in \cite{morales-espana_tight_2013} highlights the importance of a low integrality gap, demonstrating a significantly reduced solution time. This motivates a thorough polyhedral analysis.

In a Unit Commitment Problem, the operational state of each unit~$i \in \Units$ is represented by a binary vector~$\onOff(i,) \in \set{0,1}^T$, where for each period~$t \in \Periods$, $\onOff(i,t) = 1$ iff the unit is online. As defined in \eqref{equation:SCS_DiscreteStartupCostSum}, we denote by $\DiscreteStartupCostSum$ the summed start-up cost function, such that $\DiscreteStartupCostSum(\onOff(i,))$ equals the total start-up costs incurred by unit~$i$ in the course of~$\onOff(i,)$. Our major result is a correspondence between the facets of $\conv(\epi(\DiscreteStartupCostSum))$ and binary trees of size~$T$ for concave start-up cost functions~$\StartupCostFunction(i,)$, which is bijective if $\StartupCostFunction(i,)$ is strictly concave. We derive the exponential class of \emph{binary tree inequalities} (BTIs), which together with the trivial inequalities $0 \leq \onOff(i,t) \leq 1$, define an \hrepresentation of $\conv(\epi(\DiscreteStartupCostSum))$, and give an exact linear separation algorithm for the BTIs.

This result is important, both from the theoretical and practical point of view, since
\begin{itemize*}
  \item the BTIs, and in particular the proof of~\cref{result:SCS_FracPointsOnFacet}, strengthen our understanding of the start-up costs,
  \item the existing start-up cost models are extended formulations of $\conv(\epi(\DiscreteStartupCostSum))$, and may therefore be analyzed and tightened using the BTIs, and
  \item separated in a cutting plane approach, the BTIs reduce the integrality gap of the current state-of the art models (\cite{carrion_computationally_2006,simoglou_optimal_2010}) to match those of the the temperature model (\cite{silbernagl_improving_2014}).
\end{itemize*}

\section{Preliminaries}
\label{section:IntroductionSpecific}

Every unit incurs costs on start-up, mainly due to the need for reheating and extra maintenance to mitigate the effects of thermal stress. These start-up costs depend on the length~$L$ of the preceding offline time, and are assumed to be given for each unit~$i$ as
\begin{equation}
	\label{equation:SC_FunctionForTime}
 	\StartupCostFunction(i,): \quad \nonnegativeReals \to \nonnegativeReals, \quad L \mapsto \text{cost after offline time of length~$L$.}
\end{equation}

\noindent $\StartupCostFunction(i,)$ is usually described as \enquote{exponential} (\cite{wood_power_1996}), and defined as
\begin{equation}
	\label{equation:SC_ThermalStartupCosts}
	\forallUnits, L \in \nnreals: \qquad \StartupCostFunction(i,L) = \smash[t]{\begin{cases}
		\heatingCost(i) (1 - \exp{\heatloss(i) L}) + \fixedStartupCost(i) & \text{if $L > 0$,}\\
		0 & \text{else.}
	\end{cases}}
\end{equation}
with parameters $\heatingCost(i) > 0$, $\fixedStartupCost(i) \geq 0$, $\heatloss(i) > 0$ (\cf \cref{figure:SC_Function}).

\begin{figure}[!ht]
	\centering
	\def\fixed{8}
	\def\variable{25}
	\def\lambda{0.03}
	\def\maxT{12}
	\def\T{9}
	\def\offset{20}
	\begin{tikzpicture}[x=0.8mm,y=0.8mm,scale=\textwidthScaling]
		\tikzstyle{every node}+=[inner sep=0pt]
		\draw[->] (-2,0) -- (\T*\maxT+5, 0);
		\draw[anchor=west] (\T*\maxT+7,0) node (D) {$l$};
		\draw[->] (0,-2) -- (0, \variable+\fixed+2);
		\draw (-6,\variable+\fixed+4) node[anchor=north] {cost};
		\draw[blue!30] (0, \fixed+\variable) -- (\T*\maxT,\fixed+\variable);
		\draw[blue!30] (0, \fixed) -- (\T*\maxT,\fixed);
		\begin{pgfinterruptboundingbox}
			\draw [decorate,decoration={brace,amplitude=1.5\smm},xshift=-2]
				(0,0) -- (0, \fixed) node [anchor=east, black,midway,xshift=-2mm] {$\fixedStartupCost(i)$};
			\draw [decorate,decoration={brace,amplitude=1.5\smm},xshift=-2]
				(0,\fixed) -- (0, \variable + \fixed) node [anchor=east, black,midway,xshift=-2mm] {$\heatingCost(i)$};
		\end{pgfinterruptboundingbox}
		\draw[blue,very thick] (0,0) plot[smooth,domain=0:\T*\maxT,samples=100] (\x,{\fixed+\variable*(1-exp(-\lambda*\x))});
		\fill[blue] (0,0) circle (0.6mm);

		\begin{pgfinterruptboundingbox}
			\node[blue, anchor=west] at (\T*\maxT+5,\fixed+\variable) {$\StartupCostFunction(i,l)$};
		\end{pgfinterruptboundingbox}
	\end{tikzpicture}
	\caption{Typical exponential start-up cost function of a thermal unit}
	\label{figure:SC_Function}
\end{figure}
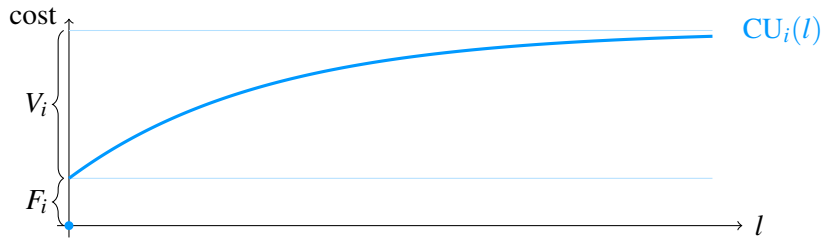

In \cite{silbernagl_improving_2014} we present how to exploit the exponential definition of the start-up costs. Here however, we only need $\StartupCostFunction(i,)$ to be non-negative, increasing, and concave, with $\StartupCostFunction(i,0) = 0$.

The modeled time range is discretized into $T$ periods of possibly varying lengths~$\periodLength(1)\mkern-3mu, \ldots, \periodLength(T)\mkern-5mu$. Furthermore, we assume the unit's offline time prior to period~$1$ to be given as $\preOffline(i)$ (\emph{p}re-model \emph{o}ffline \emph{t}ime). To simplify notation, we also discretize the start-up cost functions such that $\StartupCost(i,t,l)$ denotes the costs of a start-up of unit~$i$ in period~$t$, which was offline during the previous~$l$ periods. To this end, we define the function $\offlineLength(t,l)$ as the offline time in that situation, that is
\begin{equation}
	\label{equation:SC_OfflineLength}
	\forallPeriods,\, l \in \discrange*{0}{t-1}: \qquad \offlineLength(t,l) := \begin{cases}
		\sum_{\idx = 1}^{l} \periodLength(t-\idx) & \text{if $l < t-1$,}\\
		\sum_{\idx = 1}^{l} \periodLength(t-\idx) + \preOffline(i) & \text{if $l = t-1$.}
	\end{cases}
\end{equation}

The above case distinction differentiates between the case where the offline time lies completely within the modeled time range and the case where it stretches to include the pre-model offline time $\preOffline(i)$. The corresponding start-up costs $\StartupCost(i,t,l)$ are
\begin{equation}
	\label{equation:SC_FunctionForPeriods}
	\forallPeriods,\, l \in \discrange*{0}{t-1}: \qquad \StartupCost(i,t,l) := \StartupCostFunction{i}{\offlineLength(t,l)}.
\end{equation}

 As noted, the operational state of a unit is modeled by the variables~$\onOff(i,t) \in \set{0,1}$, where $\onOff(i,t) = 1$ iff unit~$i$ is online in period~$t$. Deriving the start-up costs in each period~$t$ from these variables yields the \emph{discrete start-up cost functions}
\begin{gather}
	\label{equation:SC_FunctionForVector}
	\forallUnitsPeriods: \quad \DiscreteStartupCost(i,t): \ \ \set{0,1}^T \to \nnreals, \ \ \onOff(i,) \mapsto \begin{cases}\StartupCost{i}{t}{\offlineNumber(t,\onOff(i,))} & \text{if } \onOff(i,t) = 1\\
		0 & \text{else,}
	\end{cases}
	\shortintertext{where}
	\label{equation:SC_OfflineNumber}
	\offlineNumber(t,\onOff(i,)) := \max \set[\big]{l \in \discrange*{1}{t-1} \given \onOff(i,t-1) = \ldots = \onOff(i,t-l) = 0}.
\end{gather}%

\noindent In \cref{figure:SC_Schedule} this relationship is depicted for an exemplary operational schedule. Note that $\offlineNumber(,)$ and $\offlineLength(,)$ follow the notational convention of $l$ and $\periodLength()$ to denote numbers of periods in lower case and time lengths in upper case.

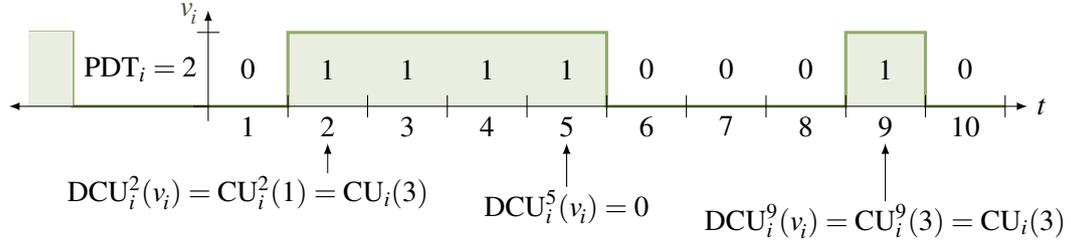
\begin{figure}[ht]
	\centering
	\begin{tikzpicture}[x=1mm,y=1mm,scale=\textwidthScaling]
	
		\def\w{10.7}
		\def\h{10}

		\begin{scope}[xscale=\w,yscale=\h]
			\drawOnOff{1/0, 2/1, 3/1, 4/1, 5/1, 6/0, 7/0, 8/0, 9/1, 10/0}

			\draw[onOffColor,very thick] (0,0) -- (-1.7,0) |- (-2.25,1);
			\fill[onOffBack] (-1.7,0) rectangle (-2.25,1);
			\node at (-0.85,0.5) {$\preOffline(i) = 2$};
		\end{scope}

		\drawTimeAxis{\w}{10}
		\draw[-latex] (0,0) -- (-2.5*\w,0);
		\drawVarAxis[]{\h}
		\begin{pgfinterruptboundingbox}
			\node[anchor=-45,onOffText] at (0,\h) {$\onOff(i,)$};
		\end{pgfinterruptboundingbox}

		\begin{scope}[xscale=\w,yscale=\h]
			\node (arrow_base) at (0,-0.5) {};
			
			\begin{scope}
				\node (cu2) at ($(onOff2.south)-(1,1.4)$) {$\DiscreteStartupCost(i,2)(\onOff(i,)) = \StartupCost(i,2,1) = \StartupCostFunction(i,3)$};
				\node[below of=onOff9, node distance=20.5\smm] (cu9) {$\DiscreteStartupCost(i,9)(\onOff(i,)) = \StartupCost(i,9,3) = \StartupCostFunction(i,3)$};
				\node[anchor=base] (cu5) at ($(cu2.base east)!0.5!(cu9.base west)$) {$\DiscreteStartupCost(i,5)(\onOff(i,)) = 0$};
				\foreach \t in {2,5,9}{
					\draw (cu\t.north -| onOff\t) edge[-latex,shorten <=-1mm] (onOff\t |- arrow_base);
				}
			\end{scope}
		\end{scope}
	\end{tikzpicture}
	\caption[Incurred start-up costs for an exemplary operational schedule]{Incurred start-up costs for an exemplary operational schedule ($\periodLengthBraces{1} = \ldots = \periodLengthBraces{T} = 1$). The operational schedule $\onOff(i,) \in \set{0,1}^T$ over the course of $T=10$ periods is shown in green. The first start-up takes place in period~2, with costs $\DiscreteStartupCost(i,2)(\onOff(i,))$ of $\StartupCostFunction(i,3)$ due to the preceding offline period with length $\periodLength(1) = 1$ and the pre-offline time $\preOffline(i)$ of 2. The only other start-up costs are incurred in period~9, after three offline periods with a total length of $\periodLength(6)+\periodLength(7)+\periodLength(8) = 3$.}
	\label{figure:SC_Schedule}
\end{figure}

In virtually all published Unit Commitment formulations which consider start-up costs, the sum of the start-up costs is minimized, motivating the definition
\begin{equation}
	\label{equation:SCS_DiscreteStartupCostSum}
	\forallUnits: \qquad \DiscreteStartupCostSum: \quad \set{0,1}^T \to \nnreals,\quad \onOff(i,) \mapsto \sum_{t=1}^T \DiscreteStartupCost(i,t)(\onOff(i,)).
\end{equation}

\newcommand{\allOnOff}{\squeeze{\onOff(1,),\ldots,\onOff(\abs{\Units},)}}

\noindent Using $\DiscreteStartupCostSum$, any Unit Commitment problem with start-up costs can be written as
\begin{equation}
	\label{equation:SCS_GeneralUnitCommitment}
	\min \set[\bigg]{ c(\allOnOff,z) + \sum_{i \in \Units} \DiscreteStartupCostSum(\onOff(i,)) \given (\allOnOff,z) \in \mathcal{F}},
\end{equation}
with additional variables~$z \in \reals^n$, feasible set~$\squeeze{\mathcal{F} \subseteq \set{0,1}^{T\abs{\Units}} \times \reals^n}$, and objective function~$c$. For example, in a typical Unit Commitment problem $z$ may stand for electricity output and cost variables, and $\mathcal{F}$ may model production limits and costs, ramping speed limits, and minimal up- and downtime. Recognize, even in case the objective~$c$ is linear, due to $\DiscreteStartupCostSum$ the extended objective is non-linear.

Since the domain~$\set{0,1}^T$ of $\DiscreteStartupCostSum$ is finite, the convex hull of its epigraph,
\begin{gather}
	\label{equation:SCS_EpigraphVRepresentation}
	\begin{aligned}
		\conv(\epi(\DiscreteStartupCostSum)) &= \conv\left(\set[\big]{(\onOff(i,),\startupCostSum) \given \onOff(i,) \in \set{0,1}^T, \startupCostSum \geq \DiscreteStartupCostSum(\onOff(i,))}\right)\\
			&= \conv(\scsVertices) + \pos(\unitVec{T+1}),
	\end{aligned}
	\intertext{is a polyhedron with unbounded direction $\unitVec{T+1}$ (where $\unitVec{T+1}$ denotes the unit vector in the direction of the last coordinate $\startupCostSum$), and with vertices}
	\label{equation:SCS_Vertices}
	\scsVertices= \set[\big]{(\onOff(i,),\DiscreteStartupCostSum(\onOff(i,))) \given \onOff(i,) \in \set{0,1}^T}.
\end{gather}
The domain~$\set{0,1}^T$ of $\DiscreteStartupCostSum$ is extremal in its convex hull~$[0,1]^T$. Therefore, using the convex hull corresponds to extending $\DiscreteStartupCostSum$ to fractional vectors~$\onOff(i,) \in [0,1]^T$, \ie to the unique function~$\StartupCostSum$ with
\vskip-1ex
\begin{equation}
	\label{equation:SCS_StartupCostSum}
	\scsEpigraph = \conv(\epi(\DiscreteStartupCostSum)).
\end{equation}
So, by introducing the variables~$\startupCostSum$, the minimization problem~\eqref{equation:SCS_GeneralUnitCommitment} can be rewritten as
\begin{equation}
	\label{equation:SCS_GeneralUnitCommitmentEpigraph}
	\min \set[\bigg]{ c(\allOnOff,z) + \sum_{\mathclap{i \in \Units}} \startupCostSum \given \forall\, i \in \mkern-2mu\Units: (\onOff(i,),\startupCostSum) \in \scsEpigraph, (\allOnOff,z) \in \mathcal{F}}.
\end{equation}
If $c$ is affine linear and $\mathcal{F}$ is expressible in a Mixed Integer Program (MIP), then \eqref{equation:SCS_GeneralUnitCommitmentEpigraph} is a MIP as well, since the polyhedral~$\scsEpigraph$ may be modeled using its \vrepresentation in \eqref{equation:SCS_EpigraphVRepresentation} (see \cite{dantzig_decomposition_1960} for details).

Due to the $2^T$ vertices of~$\scsEpigraph$ however, this would lead to an excessive model size. All existing approaches, including among others \cite{nowak_stochastic_2000,carrion_computationally_2006,simoglou_optimal_2010}, mitigate this by modeling sets which enclose $\scsEpigraph$, in \hrepresentation. In this publication, we provide an exact \hrepresentation of $\scsEpigraph$, accompanied by an exact separation algorithm with running time~$\bigO(T)$. At its core, our contribution results from four observations:
\begin{enumerate}
  \item Iteratively lifting a trivial inequality variable by variable results in facets of $\scsEpigraph$ and in coefficients which can be derived explicitly from the lifting order.
  \item Multiple lifting orders result in the same facet. Each lifted facet is already uniquely described by a special partial order, which can be expressed as a binary tree, warranting the name \emph{binary tree inequality} (BTI) for the lifted inequalities.
  \item The binary tree associated to a lifted facet readily identifies the points on this facet. Furthermore, these facets describe the complete lower boundary (with respect to the last coordinate~$\startupCostSum$) of the epigraph, proving that it does not have further non-trivial facets.
  \item For a fractional point, a suitable binary tree and the coefficients of the respective lifted facet can be determined efficiently.
\end{enumerate}

\noindent These observations are presented as follows:
\begin{itemize}
  \item in Section~\ref{section:SCS_Lifting} the result of lifting a trivial inequality is analyzed,
  \item in Section~\ref{section:SCS_BT} special notation for binary trees is introduced,
  \item in Section~\ref{section:SCS_BinaryTreeInequalities} the so-called \emph{binary tree inequalities} are derived and shown to be facet-inducing,
  \item in Section~\ref{section:SCS_HRepresentation} these inequalities are shown to complete the \hrepresentation of the epigraph of the summed start-up costs, and finally
  \item in Section~\ref{section:SCS_Separation} an exact separation algorithm is presented.
\end{itemize}
Since all these sections consider the epigraph~$\scsEpigraph$ of a single unit, we drop the unit index~$i$ in the following.
\def\defaultCompany{i}

\section{Lifting Inequalities}
\label{section:SCS_Lifting}

\providecommand{\liftNext}{{q}}

In this subsection, we observe that by lifting the trivial inequality $\startupCostSum \geq 0$, we gain facet-inducing inequalities with explicitly derivable coefficients.

The choice of~$\startupCostSum \geq 0$ as a starting point is motivated by the properties of $\DiscreteStartupCost(i,t)$, which imply that $\StartupCostSum$ is homogeneous, \ie that $\lambda\StartupCostSum(\onOff(i,)) = \StartupCostSum(\lambda \onOff(i,))$ for all $\lambda \in [0,1]$, $\onOff(i,) \in [0,1]^T$ (this follows explicitly from \cite{silbernagl_thesis} and implicitly from \cref{result:SCS_Homogeneous} later).
It holds for all $(x,y) \in [0,1]^T \mkern-4mu\times \reals$, $\lambda \in (0,1]$ that
\begin{align*}
	(\onOff(i,),\startupCostSum) \in \scsEpigraph &\equals \startupCostSum \geq \StartupCostSum(\onOff(i,)) \equals \lambda \startupCostSum \geq \lambda \StartupCostSum(\onOff(i,)) = \StartupCostSum(\lambda \onOff(i,))\\
	&\equals \lambda(\onOff(i,),\startupCostSum) \in \scsEpigraph.
\end{align*}
This means, $\scsEpigraph$ is the intersection of a cone with $[0,1]^T \times \reals$, and thus all facets of $\scsEpigraph$, except those induced by $0 \leq \onOff(i,t) \leq 1$, must also be facets of this cone. Hence, the facet-inducing inequalities are homogeneous, \ie without constant term, and since $\startupCostSum$ is not bounded from above in $\scsEpigraph$, of the kind
\begin{equation}
	\label{equation:SCS_LiftedInequality}
	\startupCostSum \geq \sum_{t \in \Periods} \bticoeff(t) \onOff(i,t).
\end{equation}

\newvar{liftSpace}{1}{\mathcal{F}^{#1}(\sigma)}
\newvar{liftVector}{2}{\onOff{#1}{#2}(\sigma)}
\newvar{liftSpaceInc}{1}{\mathcal{F}^{#1}_1(\sigma)}
Starting from $\startupCostSum \geq 0$, such facets can be identified using the standard sequential lifting method (see \eg \cite{bertsimas_optimization_2005}). For each order $\sigma: \Periods \to \Periods$, this method determines the coefficients~$\bticoeff{\sigma(\idx)}$ consecutively by considering the start-up costs for operational schedules~$\onOff(i,)$ with coordinates~$\sigma(\idx+1), \ldots, \sigma(T)$ fixed to $0$,
\begin{align*}
	\forall\, \idx \in \discrange{0}{T}: \qquad \liftSpace(\idx) &:= \set[\Big]{\onOff(i,) \in \set{0,1}^T \given \forall\, k \in \discrange*{\idx+1}{T}: \onOff{\idx}{\sigma(\idxx)} = 0}.\\
	\intertext{Here, each $\smash{\liftSpace(\idx)}$ extends $\smash{\liftSpace(\idx-1)}$ by all $\onOff(i,)$ with $\smash{\onOff{\idx}{\sigma(\idx)}} = 1$,}
	\forall\, \idx \in \discrange{1}{T}: \qquad \liftSpaceInc(\idx) &:= \liftSpace(\idx) \setminus \liftSpace(\idx-1) = \set{\onOff(,) \in \liftSpace(\idx), \onOff{}{\sigma(\idx)} = 1}.
\end{align*}
Being fulfilled by $0 \in \scsEpigraph$ with equality, $\startupCostSum \geq 0$ induces a 0-dimensional face of $\scsEpigraph$. By determining the coefficients~$\bticoeff{\sigma(\idx)}$ as
\begin{equation}
	\label{equation:SCS_LiftingCoefficient}
	\forall\, \idx \in \discrange{1}{T}: \qquad \bticoeff{\sigma(\idx)} := \min \set[\bigg]{\DiscreteStartupCostSum(\onOff(i,)) - \sum_{\idxx=1}^{\idx-1} \bticoeff{\sigma(\idxx)} \onOffBraces{i}{\sigma(\idxx)} \given \onOff(,) \in \liftSpaceInc(\idx)},
\end{equation}
the sequential lifting method produces a series of inequalities defining $\idx$-faces of $\scsEpigraph$, culminating in a facet of $\scsEpigraph$.

We claim that the minimum in \eqref{equation:SCS_LiftingCoefficient} is attained by the vectors $\liftVector(\idx,) \in \liftSpaceInc(\idx)$ with
\begin{equation}
	\label{equation:SCS_LiftingVectors}
	\forallPeriods: \qquad \liftVector(\idx,) := \liftVector(\idx-1,) + \unitVec{\idx},
\end{equation}
using $\liftVector(0,) := (0,\ldots,0)$. The coefficients of the lifted inequality would thus equate to
\begin{equation}
	\label{equation:SCS_LiftingCoefficientsClaim1}
	\begin{aligned}
		\forall\, j \in \discrange{1}{T}: \qquad \bticoeff{\sigma(\idx)} &= \DiscreteStartupCostSum(\liftVector(\idx,))-\DiscreteStartupCostSum(\liftVector(\idx-1,))\\
		&= \DiscreteStartupCostSum(\liftVector(\idx-1,) + \unitVec{\sigma(\idx)})-\DiscreteStartupCostSum(\liftVector(\idx-1,)).
	\end{aligned}
\end{equation}

To prove this, a closer look at the change of the summed start-up costs that defines $\bticoeff{\sigma(\idx)}$ is necessary. This change depends on $\liftVector(\idx-1,)$ and $\sigma(\idx)$, and not on the relative lifting order of $\sigma(1), \ldots, \sigma(\idx-1)$ and $\sigma(\idx+1), \ldots, \sigma(T)$. We express this by considering $\smash{\onOff(i,),\tildeOnOff(i,) \in \set{0,1}^T}$ and $t \in \Periods$ with $\tildeOnOff(i,) = \onOff(i,) + \unitVec{t}$ in place of $\liftVector(\idx-1,)$, $\liftVector(\idx,)$, and $\sigma(\idx)$.

\newpage
As can be seen in \cref{figure:SCS_LiftingExample}, there are at most two indices $\tau \in \Periods$ such that the start-up costs $\StartupCost{i}{\tau}{\onOff(i,)}$ and $\StartupCost{i}{\tau}{\tildeOnOff(i,)}$ differ,
\begin{enumerate*}
  		\item the index $t$ itself, and
  		\item the minimal index $\liftNext \in \discrange*{t+1}{T}$ with $\onOff(i,\liftNext) = 1$, if such an index exists.
\end{enumerate*}

These start-up costs depend on the number of offline periods immediately before and after period~$\sigma(\idx)$. The number of offline periods preceding $\sigma(\idx)$ is given by $\offlineNumber{\sigma(\idx)}{\tildeOnOff(i,)}$ (see \eqref{equation:SC_OfflineLength}). For each $\onOff(i,) \in \set{0,1}^T$, $t \in \Periods$, we define a symmetric function for the offline periods succeeding $t$ as
\begin{equation}
	\label{equation:SCS_OfflineNumberRight}
	\offlineNumberRight{t}{\onOff(i,)} := \max\set{\idx \in \discrange*{0}{T-t} \given \tildeOnOff(i,t+1) = \tildeOnOff(i,t+2) = \ldots = \tildeOnOff(i,t+\idx) = 0}.
\end{equation}

\begin{figure}[!ht]
	\centering
	\begin{tikzpicture}[x=1cm,y=1cm,scale=\textwidthScaling]
		\begin{scope}[yshift=1.5cm]
			\node[anchor=east] at (0,0.5) {$\onOff(i,)$};
			\drawOnOff{\liftingInput}
		\end{scope}

		\begin{scope}
			\node[anchor=east] at (0,0.5) {$\tildeOnOff(i,)$};
			\drawOnOff{\liftingResult};
		\end{scope}

		\begin{pgfonlayer}{background}
			\fill[red!10,yshift=-0.57cm,yscale=4,rounded corners=2ex]
			              (5,0) rectangle +(1,1)
			              (10,0) rectangle +(1,1);
		\end{pgfonlayer}
		\node (idx)  at (5.5,3) {$t$};
		\node (next) at (10.5,3) {$\liftNext$};

		\draw[->] (0,-0.3) -- (0,2.8);
		\draw[->] (-0.3,1.5) -- (12.5,1.5);
		\draw[->] (-0.3,0) -- (12.5,0) node[anchor=west] {$t$};
		\foreach \i in {1,2,...,12} {
			\node[anchor=north] (t\i) at (\i-0.5, 0) {$\i$};
			\draw (\i, 0) +(0, -0.08) -- +(0, 0.08);
			\draw (\i, 1.5) +(0, -0.08) -- +(0, 0.08);
		}

		\draw [decorate,decoration={brace,amplitude=2\smm,mirror}] ($(2.1,0 |- t1.south)+(0,1mm)$) -- ($(4.9,0 |- t1.south)+(0,1mm)$) node [midway,anchor=north,yshift=-2\smm] {$\offlineNumber{t}{\tildeOnOff(i,)}$};
		\draw [decorate,decoration={brace,amplitude=2\smm,mirror}] ($(6.1,0 |- t1.south)+(0,1mm)$) -- ($(9.9,0 |- t1.south)+(0,1mm)$) node (h1) [midway,anchor=north,yshift=-2\smm] {$\offlineNumberRight{t}{\tildeOnOff(i,)}$};
		\draw [decorate,decoration={brace,amplitude=2\smm,mirror}] ($(2.1,0 |- h1.south)+(0,1mm)$) -- ($(9.9,0 |- h1.south)+(0,1mm)$) node [midway,anchor=north,yshift=-2\smm] {$\offlineNumber{t}{\tildeOnOff(i,)} + \offlineNumberRight{t}{\tildeOnOff(i,)} + 1$};

	\end{tikzpicture}
	\caption[A step of the sequential lifting method]{A step of the sequential lifting method, from the vector $\onOff(i,)$ to the vector $\tildeOnOff(i,)$. By setting $\smash{\onOff{\idx}{t}} = 1$, the offline time of length~$\offlineNumber{t}{\tildeOnOff(i,)} + \offlineNumberRight{t}{\tildeOnOff(i,)} + 1$ is split into two offline times of lengths~$\offlineNumber{t}{\tildeOnOff(i,)}$ and $\offlineNumberRight{t}{\tildeOnOff(i,)}$, thereby changing the summed start-up costs.}
	\label{figure:SCS_LiftingExample}
\end{figure}
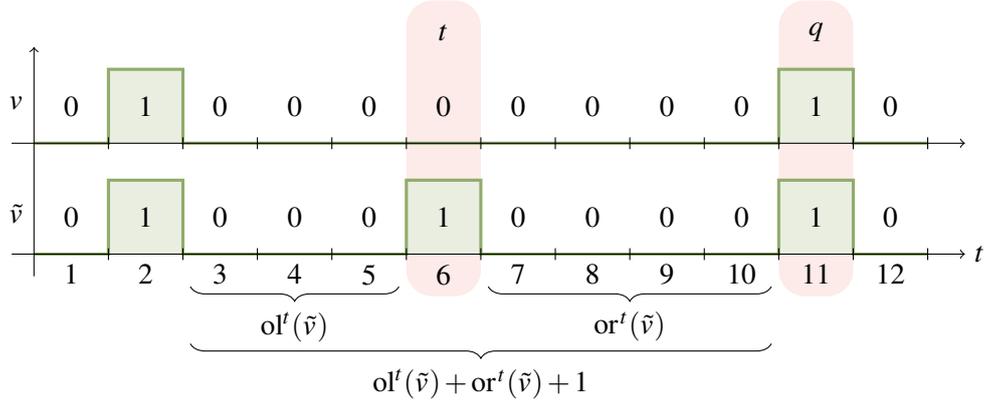

So, for two operational schedules~$\onOff(i,),\tildeOnOff(i,) \in \set{0,1}^T$ which differ solely in period~$t$, the start-up costs are unequal in periods~$t$ and $t + \offlineNumberRight{t}{\tildeOnOff(i,)}$.
	Abbreviating $l := \offlineNumber{t}{\tildeOnOff(i,)}$ and $r := \offlineNumberRight{t}{\tildeOnOff(i,)}$, we obtain
\begin{equation*}
	\DiscreteStartupCostSum(\tildeOnOff(i,)) - \DiscreteStartupCostSum(\onOff(i,))
	= \begin{cases}
		\DiscreteStartupCost(i,t)(\tildeOnOff(i,)) + \DiscreteStartupCost{i}{t + r + 1}(\tildeOnOff(i,)) - \DiscreteStartupCost{i}{t + r + 1}(\onOff(,)) & \text{ if $t + r < T$},\\
		\DiscreteStartupCost(i,t)(\tildeOnOff(i,)) & \text{ else.}
	\end{cases}
\end{equation*}
Thus, the change in the summed start-up costs depends only on the offline periods prior and after~$t$, and can be further simplified.

\begin{proposition}
	\label[proposition]{result:SCS_SumUpdate}
	For each $\onOff(i,), \tildeOnOff(i,) \in \set{0,1}^T$ and $t \in \Periods$ with $\tildeOnOff(i,) = \onOff(i,) + \unitVec{t}$, it holds
	\begin{gather*}
		\DiscreteStartupCostSum(\tildeOnOff(i,)) - \DiscreteStartupCostSum(\onOff(i,)) = \scsSumDiff{t}(\offlineNumber{t}{\tildeOnOff(i,)},\offlineNumberRight{t}{\tildeOnOff(i,)}),
		\shortintertext{where}
		\begin{aligned}
			\scsSumDiff{t}:\quad &\discrange*{0}{t-1} \times \discrange*{0}{T-t} \to \reals\\
			&(l,r) \mapsto \begin{cases}
			\StartupCost(i,t,l) + \StartupCost(i,t + r + 1, r) - \StartupCost(i,t + r + 1,l + r + 1) & \text{ if $t + r < T$},\\
			\StartupCost(i,t,l) & \text{else.}
		\end{cases}
		\end{aligned}
	\end{gather*}
\end{proposition}

The following lemma shows where the concavity of the start-up cost function~$\StartupCostFunction(i,)$ is used.
\begin{lemma}
	\label[lemma]{result:SCS_SumDiffIncreasing}
	For each $t \in \Periods$, $\scsSumDiff{t}$ is increasing in $l$ and $r$ if $\StartupCostFunction(i,)$ is concave, and strictly increasing if $\StartupCostFunction(i,)$ is strictly concave.
\end{lemma}
\begin{proof}
	Let $t \in \Periods$, $l, \tilde{l} \in \discrange*{0}{t-1}$ with $l < \tilde{l}$ and $r,\tilde{r} \in \discrange*{0}{T-t}$  with $r < \tilde{r}$ be given. Denote the period indices
	\begin{equation*}
		\liftNext := t + r + 1 \quad\text{and}\quad \tilde{\liftNext} := t + \tilde{r} + 1,
	\end{equation*}
	as seen in \cref{figure:SCS_LiftingExample}, implying $\liftNext < \tilde{\liftNext}$.  Recall that by definition, $\StartupCost(i,t,l) = \StartupCostFunction{i}{\offlineLength(t,l)}$, where $\offlineLength(t,l)$ denotes the offline length corresponding to $l$ offline periods prior to period~$t$ (see \eqref{equation:SC_FunctionForPeriods}). In \cref{figure:SCS_LiftingExample} for example, $l$ corresponds to $\offlineNumber{t}{\tildeOnOff(i,)}$ and $r$ corresponds to $\offlineNumberRight{t}{\tildeOnOff(i,)}$. So, $\offlineLength(q,l+r+1)$ is the offline time prior to~$q$ in $\onOff(i,)$, and $\offlineLength(q,r)$ and $\offlineLength(t,l)$ are the offline times prior to $q$ and $t$ in $\tildeOnOff(i,)$, respectively. Hence
	\begin{equation*}
		\offlineLength(q,l+r+1) = \offlineLength(q,r+1) + \offlineLength(t,l).
	\end{equation*}
		
	\noindent We start by proving that $\scsSumDiff{t}$ is increasing in $l$. For $t + r = T$, $\scsSumDiff{t}$ increases in~$l$ since $\StartupCost(i,t,)$ increases,
	\begin{equation*}
		\scsSumDiff{t}(\tilde{l},r) - \scsSumDiff{t}(l,r) = \StartupCost(i, t, \tilde{l}) - \StartupCost(i, t, l) \geq 0.
	\end{equation*}
	
	\noindent For $t+r < T$, we obtain 
	\begin{align*}
		\scsSumDiff{t}(\tilde{l},&r) - \scsSumDiff{t}(l,r)
		= \StartupCost(i,q,r) + \StartupCost(i,t,\tilde{l}) - \StartupCost(i,q,\tilde{l}+r+1)\\
		&\phantom{r)- \scsSumDiff{t}(l,r) =} - \Big(\StartupCost(i,q,r) + \StartupCost(i,t,l) - \StartupCost(i,q,l+r+1)\Big)\\
		&= \left(\StartupCost(i,t,\tilde{l}) - \StartupCost(i,t,l)\right) - \left(\StartupCost(i,q,\tilde{l}+r+1) - \StartupCost(i,1,l+r+1)\right)\\
		&= \left(\StartupCostFunction{i}{\offlineLength(t, \tilde{l})} - \StartupCostFunction{i}{\offlineLength(t, l)}\right) - \left(\StartupCostFunction{i}{\offlineLength(\liftNext, \tilde{l} + r + 1)} - \StartupCostFunction{i}{\offlineLength(\liftNext, l + r + 1)}\right)\\
		&= \left(\StartupCostFunction{i}{\offlineLength(t, l) + \offlineLength(t-l-1, \tilde{l} - l)} - \StartupCostFunction{i}{\offlineLength(t, l)}\right)\\
		&\phantom{=}- \left(\StartupCostFunction{i}{\offlineLength(\liftNext, l + r + 1) + \offlineLength(t-l-1, \tilde{l} - l)} - \StartupCostFunction{i}{\offlineLength(\liftNext, l + r + 1)}\right)\\
		\intertext{which, when abbreviating $x := \offlineLength(t, l)$, $y := \offlineLength(\liftNext, l + r + 1)$, and $s := \offlineLength(t-l-1, \tilde{l} - l)$, equals}
		&= \Big(\StartupCostFunction{i}{x+s} - \StartupCostFunction{i}{x}\Big) - \Big(\StartupCostFunction{i}{y+s} - \StartupCostFunction{i}{y}\Big).
	\end{align*}
	The non-negativity of the final term follows from the characterization of concave functions by subdifferentials, using that $x < y$ and $s > 0$.

	The statement that $\scsSumDiff{t}$ is increasing in $r$ follows analogously in the case of $t + \tilde{r} < T$. If $t + \tilde{r} = T$, then
	\begin{align*}
		\scsSumDiff{t}(l,\tilde{r}) - \scsSumDiff{t}(l,r)
		&= \StartupCost(i,t, l) - \Big( \StartupCost(i,t, l) + \StartupCost(i, \liftNext, r) - \StartupCost(i, \liftNext, l + r + 1)\Big)\\
			&= \StartupCost(i, \liftNext, l + r + 1) - \StartupCost(i, \liftNext, r) \geq 0.
	\end{align*}
	
	Finally, in case of a strictly concave start-up cost function $\StartupCostFunction(i,)$, all of the above inequalities are also strict.\manualqed
\end{proof}

Using $\scsSumDiff{\sigma(j)}$, our claim regarding the lifted coefficients~(see \eqref{equation:SCS_LiftingCoefficientsClaim1}) may be restated as
\def\claim{\scsSumDiff{\sigma(j)}(\offlineNumber{\sigma(\idx)}{\onOff(\idx,)},\offlineNumberRight{\sigma(\idx)}{\onOff(\idx,)})}
\begin{gather*}
	\bticoeff{\sigma(\idx)} = \claim.
\end{gather*}

\newpage
Based on \cref{result:SCS_SumDiffIncreasing}, this could be proved by induction over $\idx$. Furthermore, \cite{bertsimas_optimization_2005} show that, as $\startupCostSum \geq 0$ is valid for $\scsEpigraph$, these inequalities define facets of $\scsEpigraph$. Since the current subsection is intended to be purely motivational, we prefer to re-introduce the lifted inequalities in Section~\ref{section:SCS_BinaryTreeInequalities}, and prove that they induce facets without referring to the sequential lifting method. The necessary steps are essentially the same:
\begin{itemize}
  \item proving that the lifted inequalities are feasible requires the same arguments as showing $\bticoeff{\sigma(\idx)} \leq \claim$, and
  \item proving that these inequalities induce facets requires the same arguments as showing $\bticoeff{\sigma(\idx)} \geq \claim$.
\end{itemize}

\section{Notation for Binary Trees}
\label{section:SCS_BT}

Describing the lifted inequalities of the last subsection necessitates non-standard notation for binary trees, which is presented in this section. A binary tree is defined as an undirected, rooted tree, where each node~$t$ has at most two child nodes: a left child~$\leftChild(t)$ and a right child~$\rightChild(t)$ (\cf \cref{figure:BT_Repetition}).

Basic notation includes:
\begin{itemize*}
	\item Each binary tree $B$ has a \emph{root} denoted by $\root(B)$.
	\item The path from a node~$t$ to the root is unique. The number~$\depth(t)$ of its edges is called the \emph{depth} of~$t$, and its nodes are the \emph{ancestors} of~$t$ (including $t$ itself). Vice versa, $t$ is called a descendant of each of its ancestors.
	\item The first ancestor, \ie the node immediately succeeding $t$ on the path to the root, is called the \emph{parent}~$\parent(t)$ of $t$. Conversely, $t$ is said to be a \emph{child} of $\parent(t)$.  
	\item The subtree~$\subtree(t)$ comprising all descendants of a node~$t$ (including $t$) is the \emph{principal subtree in~$t$}. If the left/right child of $t$ exists, then its principal subtree is the \emph{left subtree}~$\leftSubtree(t)$/\emph{right subtree}~$\rightSubtree(t)$~of~$t$. All of these subtrees are binary trees as well.
	\item The number of nodes in these subtrees in $t$ are denoted by $\size(t) := \abs{\subtree(t)}$, $\leftSize(t) := \abs{\leftSubtree(t)}$ and $\rightSize(t) := \abs{\rightSubtree(t)}$.
	\item  The \emph{rank} function is the unique mapping from the nodes~$V$ of a binary tree~$B$ to $1,\ldots,\abs{V}$ such that
	\begin{equation*}
		\forall\, t \in V \setminus \set{\root(B)}: \qquad \rank(t) \smash[t]{\begin{cases}
			< \rank{\parent(t)} & \text{if $t$ is the left child of $\parent(t)$,}\\
			> \rank{\parent(t)} & \text{if $t$ is the right child of $\parent(t)$,}
		\end{cases}}
	\end{equation*}
	and such that the nodes of each principal subtree~$\subtree(t)$ are a mapped to a contiguous interval.
\end{itemize*}

\noindent The rank function has a straight-forward interpretation: When drawing a binary tree such that left (right) childs are located below left (right) of their parents, the rank numbers the nodes from left to right (see \cref{figure:BT_Repetition}). Several basic properties can be derived from this observation:
\begin{itemize}
	\item For each binary tree~$B$ and $t \in B$, it holds that
\begin{equation*}
	\set{\rank(\tau) \given \tau \in \subtree(t)} = \set{\rank(\tau) \given \tau \in \leftSubtree(t)} \cup \set{\rank(t)} \cup \set{\rank(\tau) \given \tau \in \rightSubtree(t)}.
\end{equation*}
	\item Since by definition the ranks of nodes in a subtree are contiguous, it holds that
\begin{gather}
	\label{equation:BT_SubtreeRanks}
	\begin{aligned}
		\set{\rank(\tau) \given \tau \in \leftSubtree(t)} &= \discrange*{\rank(t)-\leftSize(t)}{\rank(t)-1},\\
		\set{\rank(\tau) \given \tau \in \rightSubtree(t)} &= \discrange*{\rank(t)+1}{\rank(t)+\rightSize(t)},\text{ and}\\
		\set{\rank(\tau) \given \tau \in \subtree(t)} &= \discrange*{\rank(t)-\leftSize(t)}{\rank(t)+\rightSize(t)}.
	\end{aligned}
\end{gather}
	\item Finally, since $\leftSubtree(t) = \subtree{\leftChild(t)}$ and $\rightSubtree(t) = \subtree{\rightChild(t)}$, we have
\begin{equation}
	\label{equation:BT_ChildRank}
	\rank{\leftChild(t)} + \rightSize{\leftChild(t)} + 1 = \rank(t) = \rank{\rightChild(t)} - \leftSize{\rightChild(t)} - 1.
\end{equation}
\end{itemize}
For example, in the binary tree shown in \cref{figure:BT_Repetition} it holds that
\begin{gather*}
	\set{\rank(\tau) \given \tau \in \subtree(t)} = \underbrace{\discrange{7}{10}}_{\set{\rank(\tau) \given \tau \in \leftSubtree(t)}} \cup \underbrace{\set{11}}_{\rank(\tau)} \cup \underbrace{\set{12}}_{\set{\rank(\tau) \given \tau \in \rightSubtree(t)}}\quad\text{and}\\[2ex]
	\underbrace{1}_{\rank{\leftChild(t)}} + \underbrace{1}_{\rightSize{\leftChild(t)}} + 1 = \underbrace{3}_{\rank(t)} = \underbrace{4}_{\rank{\rightChild(t)}} - \underbrace{0}_{\leftSize{\rightChild(t)}} - 1.
\end{gather*}

\begin{figure}[!h]
	\centering\begin{tikzpicture}[scale=\textwidthScaling]
		\LiftingTrees
		\tikzstyle{desc}=[rectangle, draw=none, inner sep=0.5\smm]
		\def\shorten{0.8\smm}
		\tikzstyle{pin}=[black!80, shorten <=\shorten]

		\foreach \i/\d in {1/3,2/4,3/2,4/3,5/1} {
			\node[tree] (n\i) at (\i-1,-\d) {\i};
		}
		\foreach \i/\d in {6/0,7/4,8/3,9/4,10/2,11/1,12/2} {
			\node[tree] (n\i) at (\i,-\d) {\i};
		}

		\foreach \i/\j in {6/5,5/3,3/1,1/2,3/4,6/11,11/10,10/8,8/7,8/9,11/12} {
			\draw (n\i) edge (n\j);
		}

		\begin{pgfonlayer}{background}
			\coordinate (h11) at ($(n11)-(0,0.7ex)$);
			\fill[S1] \convexcycle{n7,n8,h11,n12,n9}{\convexRadiusB};
			\fill[S2] \convexcycle{n7,n8,n10,n9}{\convexRadiusA};
			\fill[S3] (n12) circle (\convexRadiusA);
		\end{pgfonlayer}

			\node[desc,anchor=base east] (rank) at (n1.west |- 0,-5.7) {rank:};
		\foreach \i in {1,...,12} {
			\node[desc,anchor=base] (rank\i) at (rank.base -| n\i) {\i};
			\draw  (rank\i) edge[pin,shorten >=\shorten] (n\i);
		}

		\path (n6) ++(25:1.7) node[desc] (root) {$\root(B)$};
		\draw (root) edge[pin] (n6);

		\begin{pgfinterruptboundingbox}

			\path (n3) ++ (122:1.1) node[desc] (t) {$t$};
			\path (n5) ++ (122:1.1) node[desc] (p) {$\parent(t)$};
			\path (n1 |- n3.base) node[desc,anchor=base] (l) {$\leftChild(t)$};
			\path ($(n4)+(0.7,0)$) node[desc,anchor=west] (r) {$\rightChild(t)$};
			\draw (n3) edge[pin] (t);
			\draw (n5) edge[pin] (p);
			\draw (n1) edge[pin] (l);
			\draw (n4) edge[pin] (r);
			
			\path (n11) ++ (60:1.1) node[desc] (tt) {\tau};
			\path ($(n11)!0.5!(n12)$) ++ (0,-2.15) node[desc,S1text] (S) {$\subtree(\tau)$};
			\path (n8)  ++ (140:1.2) node[desc,S2text] (L) {$\leftSubtree(\tau)$};
			\path (n12) ++ (35:1.2) node[desc,S3text] (R) {$\rightSubtree(\tau)$};
			\draw (n11) edge[pin] (tt);
			
		\end{pgfinterruptboundingbox}
		
		\draw [decorate,decoration={brace,amplitude=1.5\smm}] (rank.south -| 10.5,0) -- +(-4,0) node (leftSize) [desc,anchor=north,midway,fill=none,yshift=-2\smm] {$\leftSize(\tau) = 4$};
		\draw [decorate,decoration={brace,amplitude=1.5\smm}] (rank.south -| 12.5,0) -- +(-1,0) node (rightSize) [desc,anchor=north,midway,fill=none,yshift=-2\smm] {$\rightSize(\tau) = 1$};
		\draw [decorate,decoration={brace,amplitude=1.5\smm,aspect=0.375}] (leftSize.south -| 12.5,0) -- +(-6,0) node (size) [desc,anchor=north,pos=0.375,fill=none,yshift=-2\smm] {$\size(\tau) = 6$};
	\end{tikzpicture}
	\caption[A binary tree with nodes labeled by ranks, basic node relationships and subtrees]{A binary tree with nodes labeled by ranks from $1$ to $12$, and $\root(B) = 6$. The parent of node $t = 3$ is $\parent(t) = 5$, its left child is $\leftChild(t) = 1$ and its right child is $\rightChild(t) = 4$. The principal subtree~$\subtree(\tau)$ in $\tau=11$ is marked in blue, its left subtree~$\leftSubtree(\tau)$ in green and its right subtree~$\rightSubtree(\tau)$ in red. The respective subtree sizes are $\size(\tau)$, $\leftSize(\tau)$ and $\rightSize(\tau)$.}
	\label{figure:BT_Repetition}
\end{figure}
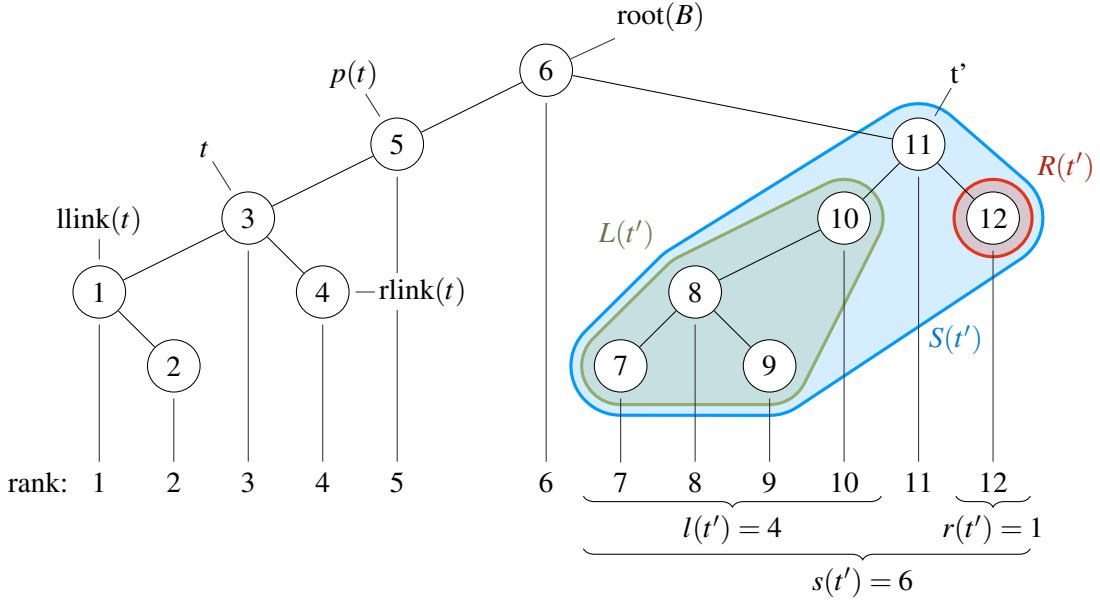

The special nature of the coefficients of the facets of $\scsEpigraph$ is best characterized by some non-standard notation, the \emph{top-left} and \emph{top-right} nodes (see \cref{figure:BT_TopRightLeft}). These nodes are defined recursively, with the root being the first top-left and top-right node. Left childs of top-left nodes are also top-left nodes, and right childs of top-right nodes are top-right nodes too. The last top-left node has rank~$1$, and the last top-right node has rank~$\abs{V}$.
\begin{definition}
	\label[definition]{definition:TopLeftRightNodes}
	For each binary tree~$B$ on $n$ nodes, define $\topleft(1) := \topright(1) := \root(B)$,
	\begin{equation*}
		\forall\, \idx \in \discrange{1}{\depth(1)}: \quad \topleft(\idx+1) := \leftChild{\topleft(\idx)}, \qquad\qquad\forall\, \idx \in \discrange{1}{\depth(n)}: \quad \topright(\idx+1) := \rightChild{\topright(\idx)},
	\end{equation*}
	where $\topleft(\idx)$ denotes the $\idx$-th top-left node and $\topright(\idx)$ denotes the $\idx$-th top-right node.
\end{definition}
\noindent Conversely, ancestors of top-left (top-right) nodes must be top-left (top-right) nodes as well.

\begin{figure}[!ht]
	\centering\begin{tikzpicture}[x=0.7cm,y=-1cm,scale=\textwidthScaling]
		\tikzstyle{every node}+=[tree]
		\tikzstyle{every edge}+=[thick]
		\tikzstyle{every pin}=[rectangle, inner sep=0pt, draw=none, pin position=10]
		
		\colorlet{topleft}{blue}
		\colorlet{topright}{red}
		\tikzset{
			topleft/.style={fill=topleft!10},
			topright/.style={fill=topright!10},
			root/.style={path picture={
				\fill[topleft!10] (path picture bounding box.south west) rectangle (path picture bounding box.north);
				\fill[topright!10] (path picture bounding box.south east) rectangle (path picture bounding box.north);
			}},
            topleftPin/.style={topleft!80!black,fill=none,pin edge={draw=topleft}},
			toprightPin/.style={topright!80!black,fill=none,pin edge={draw=topright}}
		}

		\node[topleft,pin={[topleftPin]170:$\mathllap{\topleft(\topleftN)=\topleft(3)}$}] (n1)  at ( 1,2)  {1};
		\node[topleft,pin={[topleftPin]170:$\topleft(2)$}] (n2)  at ( 2,1)  {2};
		\node (n3)  at ( 3,2)  {3};
		\node (n4)  at ( 4,3)  {4};
		\node[root,pin={[toprightPin]$\topright(1)$}, pin={[topleftPin]170:$\topleft(1)$}] (n5)  at ( 5,0)  {5};
		\node (n6)  at ( 6,3)  {6};
		\node (n7)  at ( 7,2)  {7};
		\node (n8)  at ( 8,3)  {8};
		\node[topright,pin={[toprightPin]$\topright(2)$}] (n9)  at ( 9,1)  {9};
		\node (n10)  at (10,3) {10};
		\node[topright,pin={[toprightPin]$\topright(3)\mathrlap{\ = \topright(\toprightN)}$}] (n11) at (11,2) {11};

		\draw (n5) edge (n2) edge (n9);
		\draw (n2) edge (n1) edge (n3);
		\draw (n3) edge (n4);
		\draw (n9) edge (n7) edge (n11);
		\draw (n7) edge (n6) edge (n8);
		\draw (n11) edge (n10);
	\end{tikzpicture}
	\caption[A binary tree with top-left and top-right nodes]{A binary tree with top-left nodes $\topleft(i)$ and top-right nodes $\topright(i)$. Both $\topleft(1)$ and $\topright(1)$ always equal $\root(B)$. The last top-left node $\topleft(\topleftN)$ is always the node with rank~1, and the last top-right node $\topright(\toprightN)$ is always the node with maximal rank.}
	\label{figure:BT_TopRightLeft}
\end{figure}
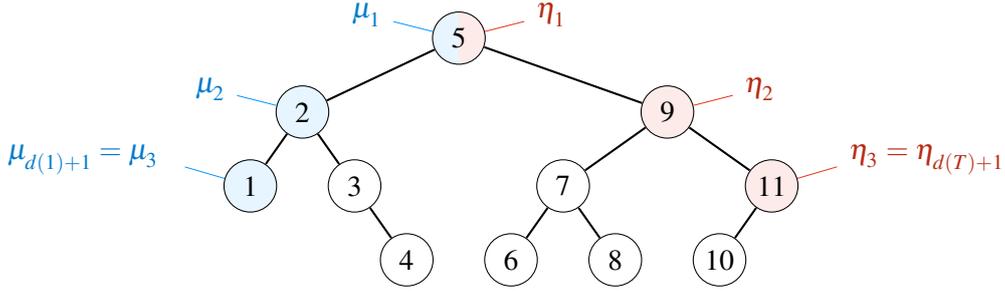

We continue by proving basic facts regarding top-right nodes, and transfer them to top-left nodes by \enquote{mirroring} the binary trees.

The node~$n$ is contained in the principal subtree of each top-right node~$\topright(\idx)$, and by definition possesses the maximal rank. Since the principal subtree of $\topright(\idx)$ spans the ranks $\discrange*{\rank(\topright{\idx})-\leftSize{\topright(\idx)}}{\rank(\topright{\idx}) + \rightSize{\topright(\idx)}}$ (see \eqref{equation:BT_SubtreeRanks}), it must hold that $\rank{\topright(i)} + \rightSize{\topright(i)} = n$. On the other hand, if the rank of a node~$t$ fulfills $\rank(t)+\rightSize(t) = n$, then $n$ lies in the right subtree of $t$. So, $t$ is an ancestor of $n$, and thus a top-right node.
\begin{proposition}
	\label[proposition]{result:TopRightOnlyRanks}
	  For each binary tree $B$ on $n$ nodes, a node $t$ is a top-right node in $B$ iff \mbox{$\rank(t) + \rightSize(t) = n$}.
\end{proposition}

This result may be extended to subtrees as well.
\begin{lemma}
	\label[lemma]{result:TopRightSubtree}
	  For each binary tree $B$ on $n$ nodes, a node $t$ is a top-right node in the left subtree of a node~$\tau$ iff $\rank(t) + \rightSize(t) + 1 = \rank(\tau)$.
\end{lemma}
\begin{proof}
	By \eqref{equation:BT_SubtreeRanks}, the nodes in the left subtree~$\leftSubtree(\tau)$ have ranks $\discrange*{\rank(\tau)-\leftSize(\tau)}{\rank(\tau)-1}$. Thus, the rank function~$\rankT(\leftSubtree{\tau},)$ of $\leftSubtree(\tau)$, which must keep the same order as $\rank()$ and ranges from $1$ to $\leftSize(\tau)$ equates
	\begin{equation*}
		\rankT{\leftSubtree{\tau}}{t} = \rank(t) - \rank(\tau) + \leftSize(\tau)+1.
	\end{equation*}
	Hence, each node~$t \in \leftSubtree(\tau)$ is a top-right node in $\leftSubtree(\tau)$ iff $\rankT{\leftSubtree{\tau}}{t} + \rightSize(t) = \abs{\leftSubtree(\tau)} = \leftSize(\tau)$ (see \cref{result:TopRightOnlyRanks}), which may be rewritten as
	\begin{equation*}
		\rank(t) + \rightSize(t) + 1 = \rankT{\leftSubtree{\tau}}{t} + \rank(\tau) - \leftSize(\tau) - 1  + \rightSize(t) + 1 = \rank(\tau).
	\end{equation*}
	
	Assume now that there exists a node~$t$ of $B$ with $t \notin \leftSubtree(\tau)$ but $\squeeze{\rank(t) = \rank(\tau) - \rightSize(t) - 1}$. Then $\rank(t) < \rank(\tau)$, and $t \notin \subtree(\tau)$. On the other hand, $\rank(\tau) > \rank(t) + \rightSize(t)$ and hence $\tau \notin \subtree(t)$.
	
	Choose the first common ancestor $t$ of $t$ and $\tau$, which due to $t \notin \subtree(\tau)$ and $\tau \notin \subtree(t)$ equals neither $t$ nor $\tau$. By choice of $t$ it holds that $t \in \leftSubtree(t)$ and $\tau \in \rightSubtree(t)$, and thus
	\begin{equation*}
		\rank(t) + \rightSize(t) < \rank(t) < \rank(\tau) - \leftSize(\tau),
	\end{equation*}
	a contradiction to $\rank(t) + \rightSize(t) + 1 = \rank(\tau)$.\manualqed
\end{proof}

Note that since $\rank(t)+\rightSize(t) + 1 \in \discrange*{\rank(t)+1}{n+1}$, $t$ must be either a top-right node or a top-right node in the left subtree of some node~$\tau$.

\begin{samepage}
\begin{corollary}
	\label[corollary]{result:TopRightRanks}
	For each binary tree on $n$ nodes, a node $t$ is
	\begin{itemize}
		\item a top-right node iff $\rank(t) + \rightSize(t) = n$, and
		\item a top-right node in the left subtree of the node~$\tau$ iff $\rank(\tau) = \rank(t) + \rightSize(t) + 1$ iff $\rank(t) + \rightSize(t) < n$.
	\end{itemize}
\end{corollary}
\end{samepage}

Let $\overline{B}$ denote the \enquote{mirrored} version of a binary tree $B$ of size $n$, \ie the binary tree that results from exchanging the left and right childs of each node. The depth, parent and subtree size functions remain unchanged ($\depthT(\overline{B},) = \depthT(B,)$, $\parentT(\overline{B},) = \parentT(B,)$, $\sizeT(\overline{B},) = \sizeT(B,)$), the left and right subtree size functions are exchanged ($\rightSizeT(\overline{B},) = \leftSizeT(B,)$, $\leftSizeT(\overline{B},) = \rightSizeT(B,)$), and the rank function is mirrored ($\rankT(\overline{B},t) = n+1-\rankT(B,)$). By applying this mirroring to the properties of the top-right nodes, we can derive equivalent properties of the top-left nodes. 

\begin{corollary}
	\label[corollary]{result:TopLeftRanks}
	For each binary tree, a node $t$ is
	\begin{itemize}
		\item a top-left node iff $\rank(t) = \leftSize(t) + 1$, and
		\item a top-left node in the right subtree of the node~$\tilde{l}$ with $\rank(\tilde{l}) = \rank(t) - \leftSize(t) - 1$ iff $\rank(t) > \leftSize(t) + 1$.
	\end{itemize}
\end{corollary}

The root of a binary tree on $n$ nodes is both a top-right and top-left node, and thus its rank equals $\rank{\root(B)} = n - \rightSize(r) = \leftSize(r) + 1$.

\section{The Binary Tree Inequalities}
\label{section:SCS_BinaryTreeInequalities}

\def\oleq{\preccurlyeq}

In this subsection, we show that all lifted inequalities correspond in a one-to-one way to binary trees, which motivates naming them \emph{binary tree inequalities (BTIs)}. Together with $0 \leq \onOff(i,t) \leq 1$, these inequalities induce all facets of $\scsEpigraph$.

We start with an example where $\bticoeff(4)$ and $\bticoeff(9)$ are lifted in both possible orders,
\begin{itemize*}
  \item $\bticoeff(4)$ before $\bticoeff(9)$ with intermediate vector $\onOff(\idx,)$, and
  \item $\bticoeff(9)$ before $\bticoeff(4)$ with intermediate vector $\tildeOnOff(\idx,)$ (\cf \cref{figure:SCS_LiftingExampleStripped}).
\end{itemize*}

\begin{figure}[!ht]
	\centering
	\def\extraScale{0.9}
	\begin{tikzpicture}[x=1cm,y=1cm,scale=\textwidthScaling]
		\def\doit#1{
			\node[anchor=east] at (0,0.5) {$#1$};
			\draw[->] (0,-0.3) -- (0,1.1);
			\draw[->] (-0.3,0) -- (12.5,0) node[anchor=west] {$t$};
			\foreach \i in {1,2,...,12} {
				\node[anchor=north] at (\i-0.5, 0) {$\i$};
				\draw (\i, 0) +(0, -0.08) -- +(0, 0.08);
			}
		}
		
		\def\jo#1{
			\fill[red!20] (2,0) rectangle +(1,1);
			\fill[blue!20] (4,0) rectangle +(1,1);
			\draw [decorate,decoration={brace,amplitude=1.5\smm},yshift=0.5mm] (2,1) -- +(1,0) node [midway,anchor=south,yshift=1mm] {$\offlineNumber{4}{#1}$};
			\draw [decorate,decoration={brace,amplitude=1.5\smm},yshift=0.5mm] (4,1) -- +(1,0) node [midway,anchor=south,yshift=1mm] {$\offlineNumberRight{4}{#1}$};
		}

		\def\je#1{
			\fill[red!20] (6,0) rectangle +(2,1);
			\fill[blue!20] (9,0) rectangle +(1,1);
			\draw [decorate,decoration={brace,amplitude=1.5\smm},yshift=0.5mm] (6,1) -- +(2,0) node [midway,anchor=south,yshift=1mm] {$\offlineNumber{9}{#1}$};
			\draw [decorate,decoration={brace,amplitude=1.5\smm},yshift=0.5mm] (9,1) -- +(1,0) node [midway,anchor=south,yshift=1mm] (h) {$\offlineNumberRight{9}{#1}$};
		}
		
		\begin{scope}[yshift=0mm,scale=\extraScale]
			\drawOnOff{\liftingResult}
			\doit{\onOff(\idx-1,)}
			
			\coordinate (lifting1) at (0,0);
		\end{scope}

		\begin{scope}[xshift=-8mm, yshift=-24mm,scale=0.7*\extraScale]
			\jo{\onOff(\idx,)}
			
			\drawOnOff{1/0,2/1,3/0,4/1,5/0,6/1,7/0,8/0,9/0,10/0,11/1,12/0}
			
			\doit{\onOff(\idx,)}

			\coordinate (liftingA) at (6,0);
		\end{scope}
		
		\begin{scope}[xshift=43mm,yshift=-41mm,scale=0.7*\extraScale]
			\je{\tildeOnOff(\idx,)}
			
			\drawOnOff{1/0,2/1,3/0,4/0,5/0,6/1,7/0,8/0,9/1,10/0,11/1,12/0}
			\doit{\tildeOnOff(\idx,)}

			\coordinate (liftingB) at (10,0);
		\end{scope}

		\begin{scope}[yshift=-66mm,scale=\extraScale]
			\jo{\onOff(\idx+1,)}
			\je{\onOff(\idx+1,)}
			
			\drawOnOff{1/0,2/1,3/0,4/1,5/0,6/1,7/0,8/0,9/1,10/0,11/1,12/0}
			\doit{\onOff(\idx+1,)}

			\coordinate (lifting3) at (0,0);
		\end{scope}
		
		\begin{scope}[thick]
			\tikzset{every node/.append style={rectangle,inner sep=2mm}}
			
			\draw[->,shorten <=6.5\smm,shorten >=9\smm] (lifting1 -| liftingA) to node[pos=0.5,anchor=west] {lifting $\bticoeff(4)$} (liftingA);
			
			\draw[->,shorten <=6.5\smm, shorten >=20.5\smm] (liftingA) to node[pos=0.38,anchor=east] {lifting $\bticoeff(9)$}  (lifting3 -| liftingA);
			
			\draw[->,shorten <=6.5\smm,shorten >=17\smm] (lifting1 -| liftingB) to node[pos=0.45,anchor=west] {lifting $\bticoeff(9)$}   (liftingB);

			\draw[->,shorten <=5.3\smm,shorten >=12\smm] (liftingB) to node[pos=0.35,anchor=west] {lifting $\bticoeff(4)$} (lifting3 -| liftingB);
		\end{scope}
	\end{tikzpicture}
	\caption{Lifting coefficients~$\protect\bticoeff(4)$ and $\protect\bticoeff(9)$ in both orders, with intermediate vectors.}
	\label{figure:SCS_LiftingExampleStripped}
\end{figure}
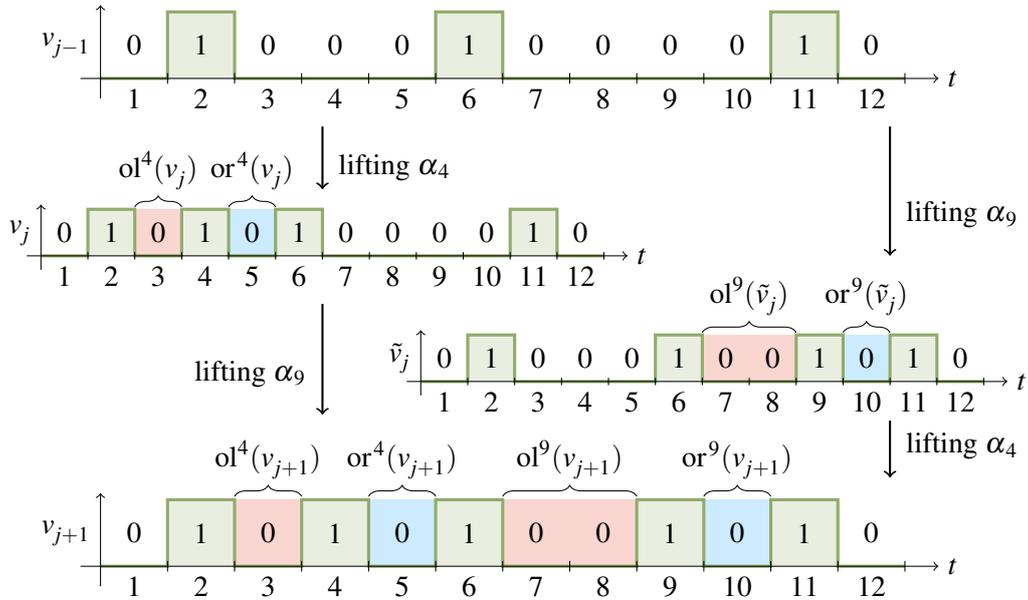

\noindent If the coefficient~$\bticoeff(6)$ is already lifted, the relative order in which $\bticoeff(4)$ and $\bticoeff(9)$ are lifted does not influence the period counts,
\begin{alignat*}{2}
	\offlineNumber{4}{\onOff(\idx,)} &= \offlineNumber{4}{\onOff(\idx+1,)} &\qquad\text{and}\qquad \offlineNumberRight{4}{\onOff(\idx,)} &= \offlineNumberRight{4}{\onOff(\idx+1,)},\\
	\offlineNumber{9}{\tildeOnOff(\idx,)} &= \offlineNumber{9}{\onOff(\idx+1,)} &\qquad\text{and}\qquad \offlineNumber{9}{\tildeOnOff(\idx,)} &= \offlineNumber{9}{\onOff(\idx+1,)}.
\end{alignat*}
Thereby, the values of the lifted coefficients~$\bticoeff(4)$ and $\bticoeff(9)$ are equal for both cases. This holds in general: As soon as a coefficient $\bticoeff(t)$ has been lifted, for each subsequently lifted coefficient~$\bticoeff(\tau)$ with $\tau < t$ we have $\offlineNumberRight{\tau}{\onOff(i,)} < t - \tau$. Thus, the period counts, and by extension the coefficient~$\bticoeff(\tau)$, do not depend on the lifting order of coefficients~$\bticoeff(\tilde{t})$ with ${\tilde{t}} > t$. Analogously if $t' > t$, the coefficient~$\bticoeff(\tau)$ does not depend on the lifting order of coefficients~$\bticoeff(\tilde{t})$ with ${\tilde{t}} < t$.

\def\a{{t_1}}
\def\b{{t_2}}

A lifting order~$\sigma$ corresponds to a linear order~$\oleq_\sigma$ on $\Periods$ with
\begin{equation*}
 	\forall\, \a, \b \in \Periods: \qquad \a \oleq_\sigma \b \quad:\Leftrightarrow\quad \sigma^{-1}(\a) \leq \sigma^{-1}(\b).
\end{equation*}
As argued above, the lifted coefficients~$\bticoeff(t_1),\bticoeff(t_2)$ do not depend on whether $t_1 \oleq_\sigma t_2$ or $t_2 \oleq_\sigma t_1$ if
\begin{equation*}
	\exists\ t_3 \in (t_1,t_2) \cup (t_1,t_2),\quad t_3 \oleq_\sigma t_1, \quad t_3 \oleq_\sigma t_2.
\end{equation*}
Eliminating such relationships from $\oleq_\sigma$ yields a partial order~$\oleq_\sigma'$ which fully determines the lifted inequality, \ie each linearization of this partial order leads to the same lifted inequality. \cref{figure:SCS_PermutationAndTree} shows an exemplary partial order (twice, as a Hasse-diagram) and two possible linearizations~$\oleq_{\sigma}$ and $\oleq_{\overline{\sigma}}$ represented by the permutations~$\sigma$ and $\overline{\sigma}$.

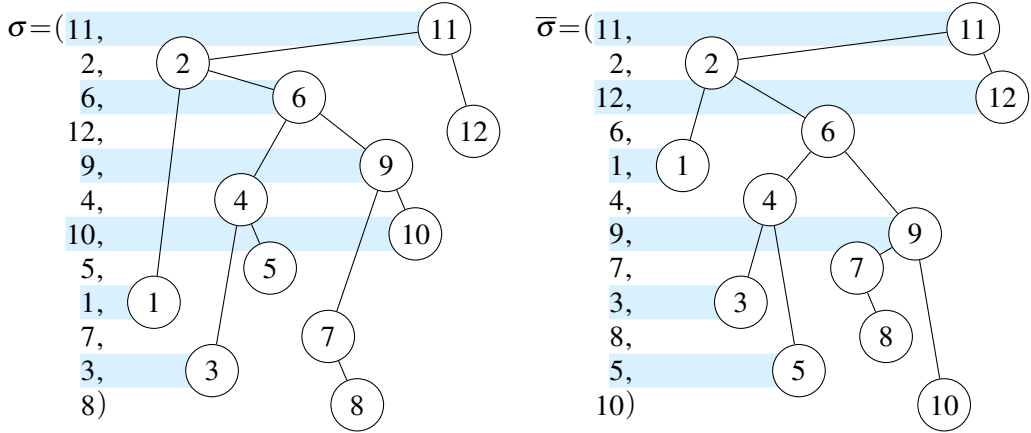
\begin{figure}[htb]
	\centering
	\begin{tikzpicture}[x=3.9mm,y=4.625mm,scale=\textwidthScaling]
		\tikzstyle{permutation}=[draw=none,inner sep=0pt,rectangle,fill=none]
		\tikzstyle{permutation_to_tree}=[blue!15]
		
		\newcommand{\drawPermTree}[2]{
			\foreach \k/\s in {#1} {
				\node[tree] (v\k) at (\k, -\s) {\k};

				\node[permutation,anchor=east] (t\s) at (-1, -\s) {\k};
				\ifthenelse{\s < 12}{
					\node[permutation, anchor=base west] at (t\s.base east) {,};
				}{}
				\begin{pgfonlayer}{background}
					\ifthenelse{\isodd{\s}}{
 						\fill[permutation_to_tree] ($(t\s.west) + (0,0.5)$) -| (v\k.center)  |- ($(t\s.west) - (0,0.5)$) -- ($(t\s.west) + (0,0.5)$);
					}{}
				\end{pgfonlayer}
			}
		
			\node[permutation,anchor=east] at (t1.west) {$\squeeze{#2 =(}$};
			\node[permutation,anchor=west] at (t12.east) {$)$};

			\liftingTreeEdges
		}
		
		\expandafter\drawPermTree\expandafter{\liftingA}{\sigma}
		\begin{scope}[xshift=7.1cm]
			\expandafter\drawPermTree\expandafter{\liftingB}{\overline{\sigma}}
		\end{scope}
	\end{tikzpicture}
	\caption{A partial order determining the lifted coefficients (twice, as a Hasse-diagram) and two possible linearizations (as permutations~$\sigma$ and $\overline{\sigma}$) leading to the same lifted inequality.}
	\label{figure:SCS_PermutationAndTree}
\end{figure}

As \cref{figure:SCS_PermutationAndTree} suggests, the Hasse-diagrams of these partial orders are binary trees, where each coefficient~$\bticoeff(t)$ corresponds to a node~$t$. The partial order then prescribes that each coefficient~$\bticoeff(t)$ must be lifted before the coefficients associated with the descendants of node~$t$.

Observe that for a lifting order~$\sigma$ and the corresponding linear order~$\oleq_\sigma$, the vectors~$\onOff(\idx,)$ encountered in the lifting process by definition \eqref{equation:SCS_LiftingVectors} fulfill
\begin{equation*}
	\forall\, \idx \in \Periods, t \in \Periods: \qquad \onOff(j,t) = \begin{cases}
		1 & \text{if $t \oleq_\sigma \sigma(\idx)$,}\\
		0 & \text{else.}
	\end{cases}
\end{equation*}
Therefore, the offline lengths~$\offlineNumber{\sigma(\idx)}{\onOff(\idx,)}$ and $\offlineNumberRight{\sigma(\idx)}{\onOff(\idx,)}$, which define $\bticoeff{\sigma(\idx)}$, equate to
\begin{align*}
	\offlineNumber{\sigma(\idx)}{\onOff(\idx,)} &= \max \set[\big]{l \in \discrange*{0}{\sigma(\idx)-1}\: \given \forall\, t \in \discrange*{\sigma(\idx)-l}{\sigma(\idx)-1}: t \oleq_\sigma \sigma(\idx)}, \text{ and}\\
	\offlineNumberRight{\sigma(\idx)}{\onOff(\idx,)} &= \max \set[\big]{l \in \discrange*{0}{T-\sigma(\idx)} \given \forall\, t \in \discrange*{\sigma(\idx)+1}{\sigma(\idx)+l}: t \oleq_\sigma \sigma(\idx)}.
\end{align*}
As argued, these lengths remain unchanged when replacing $\oleq_\sigma$ by the corresponding partial order~$\oleq_\sigma'$. Furthermore, their above representation shows that they can be derived from the Hasse-diagram of $\oleq_\sigma'$: The size~$\leftSize{\sigma(\idx)}$ of the left subtree of $\sigma(\idx)$ equals $\offlineNumber{\sigma(\idx)}{\onOff(\idx,)}$ and the size~$\rightSize{\sigma(\idx)}$ of the right subtree of~$\sigma(j)$ equals $\offlineNumberRight{\sigma(\idx)}{\onOff(\idx,)}$ (see \cref{figure:SCS_LiftingFromTreeExample}).

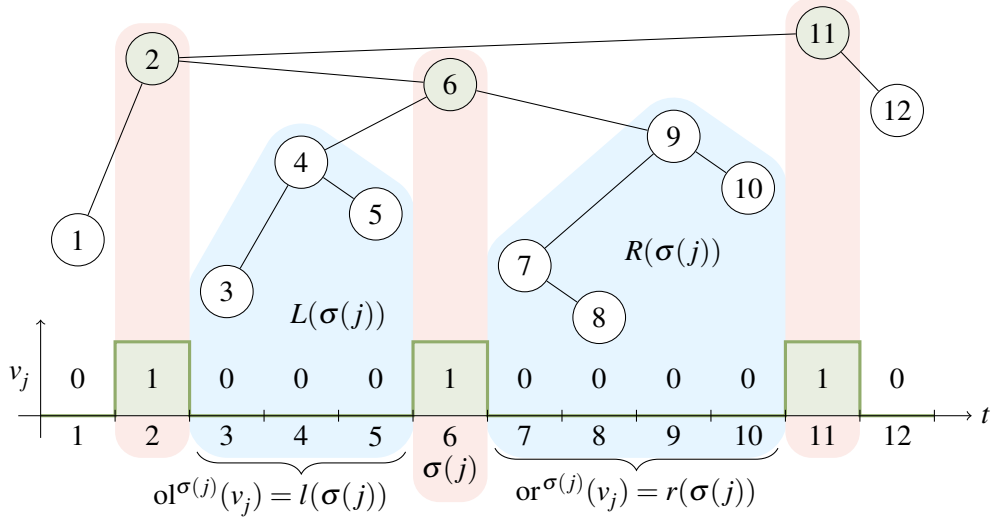
\begin{figure}[htb]
	\providecommand{\liftNext}{{q}}
	\centering
	\begin{tikzpicture}[x=1cm,y=1cm,scale=\textwidthScaling]
		\tikzstyle{subtree}=[fill=blue!10, rounded corners=\convexRadiusA]

		\begin{scope}[yscale=0.35,yshift=158mm,xshift=-5mm]
			\foreach \k/\s in \liftingAunlifted {
				\node[tree] (v\k) at (\k, -\s) {\k};
			}
			\foreach \k/\s in \liftingAlifted {
				\node[tree, lifted] (v\k) at (\k, -\s) {\k};
			}

			\liftingTreeEdges
		\end{scope}

		\drawOnOff{\liftingResult}

		\node[anchor=east] at (0,0.5) {$\onOff(\idx,)$};
		\draw[->] (0,-0.3) -- (0,1.3);
		\draw[->] (-0.3,0) -- (12.5,0) node[anchor=west] {$t$};
		\foreach \i in {1,2,...,12} {
			\node[anchor=north] (t\i) at (\i-0.5, 0) {$\i$};
			\draw (\i, 0) +(0, -0.08) -- +(0, 0.08);
		}

		\begin{pgfonlayer}{background}
			\fill[subtree,rounded corners=0pt] \convexcycle{v3,v4,v5}{\convexRadiusA};
			\fill[subtree,rounded corners=0pt] \convexcycle{v7,v9,v10,v8}{\convexRadiusA};
			\fill[subtree] ($(v3.center) - (\convexRadiusA,0)$) |- (4.5,-0.57) -| ($(v5.center)+(\convexRadiusA,0)$);
			\fill[subtree] ($(v7.center) - (\convexRadiusA,0)$) |- (7,-0.57) -| ($(v10.center)+(\convexRadiusA,0)$);
		\end{pgfonlayer}
		\node[anchor=north] at (4,0 |- v8.north) {$\leftSubtree{\sigma(\idx)}$};
		\node at ($(v8)!0.5!(v10)$) {$\rightSubtree{\sigma(\idx)}$};
		\draw [decorate,decoration={brace,amplitude=3\smm,mirror}] ($(2.1,0 |- t1.south)+(0,0.5mm)$) -- ($(4.9,0 |- t1.south)+(0,0.5mm)$) node [midway,xshift=-4mm,anchor=north,yshift=-2\smm] (subtext) {$\offlineNumber{\sigma(\idx)}{\onOff(\idx,)} = \leftSize(\sigma(\idx))$};
		\draw [decorate,decoration={brace,amplitude=3\smm,mirror}] ($(6.1,0 |- t1.south)+(0,1mm)$) -- ($(9.9,0 |- t1.south)+(0,1mm)$) node (h1) [midway,anchor=north,yshift=-2\smm] {$\offlineNumberRight{\sigma(\idx)}{\onOff(\idx,)} = \rightSize(\sigma(\idx))$};
		\node[anchor=north] (idx)  at (v6 |- 0,-0.4) {$\sigma(\idx)$};
		\begin{pgfonlayer}{background}
			\fill[red!10,yshift=-0.57cm,rounded corners=2ex]
			              ($(idx.south) - (0.5,0.1)$) rectangle ($(v6.north) + (0.5,0.12)$)
			              (10,0) rectangle ($(v11.north) + (0.5,0.12)$)
			              (1,0) rectangle ($(v2.north) + (0.5,0.12)$);
		\end{pgfonlayer}
		
	\end{tikzpicture}
	\caption[The lifting step as shown in \cref{figure:SCS_LiftingExample}, with a possible lifting order represented by a binary tree]{The lifting step as shown in \cref{figure:SCS_LiftingExample}, with a possible lifting order represented by a binary tree $B$. The offline lengths~$\offlineNumber{\sigma(\idx)}{\onOff(\idx,)}$ and $\offlineNumberRight{\sigma(\idx)}{\onOff(\idx,)}$ adjacent to period~$\sigma(\idx)$ are equal to the sizes of the left and right subtree of node~$\sigma(\idx)$ in $B$.}
	\label{figure:SCS_LiftingFromTreeExample}
\end{figure}

A coefficient~$\bticoeff(t)$ therefore solely depends on the partial order~$\oleq_\sigma'$ and may be expressed as
\begin{equation*}
	\bticoeff(t) = \scsSumDiff{t}(\leftSize{t},\rightSize{t}),
\end{equation*}
where $t = \sigma(j)$, and $\leftSize(t), \rightSize(t)$ are determined by the partial order.

Each of the start-up cost terms in $\scsSumDiff{t}(\leftSize(t),\rightSize(t))$ corresponds to the cost incurred when starting up after being offline during the periods contained in either the left subtree~$\leftSubtree(t)$ of~$t$, the right subtree~$\rightSubtree(t)$ of~$t$, or the principal subtree~$\subtree(t)$ of~$t$ (see \eqref{equation:BT_SubtreeRanks}),
\begin{equation*}
	\scsSumDiff{t}(\leftSize(t),\rightSize(t)) = \underbrace{\StartupCost{i}{t}{\leftSize(t)}}_{\text{offline in $\leftSubtree(t)$}}
	+ \underbrace{\StartupCost{i}{t+\rightSize(t)+1}{\rightSize(t)}}_{\text{offline in $\rightSubtree(t)$}}
	- \underbrace{\StartupCost{i}{t+\rightSize(t)+1}{\leftSize(t)+1+\rightSize(t)}}_{\text{offline in $\subtree(t)$}}
\end{equation*}

To simplify the notation, we generalize the definition (see \eqref{equation:SC_FunctionForPeriods}) of the offline length~$\offlineLength(,)$: for each $B \in \scsBinaryTrees$ and $t \in \Periods$, we abbreviate
\begin{gather}
	\label{equation:SC_OfflineLengthTree}
	\offlineLengthTree{\subtree(t)} := \offlineLength{t+\rightSize(t)+1}{\size(t)},\\[-1.5ex]
	\intertext{and consequently, by \eqref{equation:BT_SubtreeRanks}}
	\label{equation:SC_OfflineLengthTreeLeftRight}
	\offlineLengthTree{\leftSubtree(t)} = \offlineLength{t}{\leftSize(t)} \quad\text{and}\quad \offlineLengthTree{\rightSubtree(t)} = \offlineLength{t+\rightSize(t)+1}{\rightSize(t)}.
\end{gather}
Since by definition $\StartupCost(i,t,l) = \StartupCostFunction{i}{\offlineLength(t,l)}$, it follows that
\begin{equation}
	\label{equation:SCS_SumUpdateDiffTree}
	\scsSumDiff{t}(\leftSize(t),\rightSize(t)) = \StartupCostFunction(i,\offlineLengthTree{\leftSubtree(t)}) + \StartupCostFunction(i,\offlineLengthTree{\rightSubtree(t)}) - \StartupCostFunction{i}{\offlineLengthTree{\subtree(t)}}.
\end{equation}

\newpage
\begin{definition}
	\label[definition]{definition:SCS_BinaryTreeInequality}
	A \emph{rank-labeled} binary tree is a binary tree~$B$ on nodes~$\discrange{1}{n}$, where $n \in \naturals$, and $\rank(i) = i$ for all $i \in \discrange{1}{n}$. Let $\scsBinaryTrees$ denote the family of all rank-labeled binary trees on $\Periods$. For each $B \in \scsBinaryTrees$, we define the \emph{binary tree inequality (BTI)} as
	\begin{equation*}
		\startupCostSum \geq \sum_{t \in \Periods} \scsSumDiff{t}( \leftSize(t), \rightSize(t)) \, \onOff(i,t),%
	\end{equation*}
	using the sizes~$\leftSize(t)$ or $\rightSize(t)$ of the left or right subtrees of~$t$, respectively, and $\scsSumDiff{t}$ as defined in \cref{result:SCS_SumUpdate}.
\end{definition}

In the following, we confirm that the binary tree inequalities, together with the trivial inequalities~$0 \leq \onOff(i,t) \leq 1$, define all non-trivial facets of $\scsEpigraph$ by proving that
\begin{itemize*}
  \item they are feasible (\cref{result:SCS_EpigraphInPolyhedron}),
  \item they are fulfilled with equality by all points $(\onOff(\idx,),\DiscreteStartupCostSum(\onOff(\idx,)))$ encountered during the lifting process (\cref{result:SCS_PointsOnFacet}), and
  \item these points are linearly independent (\cref{result:SCS_Facets}).
\end{itemize*}
Moreover, we will show that all points not in $\scsEpigraph$ can be separated by the BTIs or the trivial inequalities $0 \leq \onOff(i,t) \leq 1$ in $\bigO(T)$.

In the following, we need to put the vertices $(\onOff(i,),\DiscreteStartupCostSum(\onOff(i,)))$ of the epigraph $\scsEpigraph$ into relation with the binary trees $B \in \scsBinaryTrees$.
\begin{definition}
	\label[definition]{definition:SCS_Induced_Subtree}
	For each $B \in \scsBinaryTrees$ with edges $E$ and $\onOff(i,) \in \set{0,1}^T$, define $\oneSubtree{\onOff(i,)}$ as the subgraph of $B$ induced by the 1-entries of $\onOff(i,)$,
	\begin{equation*}
		\oneSubtree{\onOff(i,)} := (V_S, E_S) \quad\text{where}\quad  V_S := \set[\big]{t \in \Periods \given \onOff(i,t) = 1} \quad\text{and}\quad E_S := \set[\big]{e \in E \given e \subseteq V_S}.
	\end{equation*}
\end{definition}
Note that in general, $\oneSubtree{\onOff(i,)}$ is not connected.

The proof that the binary tree inequalities are valid implicitly follows from the concavity of $\StartupCostFunction(i,)$, which is exploited through the monotonicity of $\scsSumDiff{t}$ (\cref{result:SCS_SumDiffIncreasing}).

\begin{lemma}
	\label[lemma]{result:SCS_EpigraphInPolyhedron}
	For each $B \in \scsBinaryTrees$, the BTI induced by $B$ is valid for $\scsEpigraph$.
\end{lemma}
\begin{proof}
	Due to the convexity of $\scsEpigraph$ and since a binary tree inequality bounds $\startupCostSum$ only from below, it suffices to prove that all vertices $(\onOff(i,),\DiscreteStartupCostSum(\onOff(i,))) \in \scsVertices$ of $\scsEpigraph$ fulfill all BTIs. To do so, for each $B \in \scsBinaryTrees$, we prove that its induced BTI is valid by induction over the number of nodes $n$ in its induced subtree $\oneSubtree{\onOff(i,)}$.
	
	For $n = 0$, the only such point is $0 \in \scsVertices$. Since each BTI is homogeneous, $0$ fulfills all of them with equality. For $n \geq 1$, choose a leaf $t$ in $\oneSubtree{\onOff(i,)}$, \ie a node such that the subtree of $t$ in $B$ does not contain any other nodes from $\oneSubtree{\onOff(i,)}$. Define the vector $\tildeOnOff(,)$ as
	\begin{align*}
		\tildeOnOff(i,\tau) &:= \begin{cases}
			\onOff(i,\tau) & \text{if } \tau \neq t,\\
			0 & \text{if } \tau = t,
		\end{cases},
	\end{align*}
	differing from $\onOff(i,)$ only in period $t$. By \cref{result:SCS_SumUpdate}, we get
	\begin{equation*}
		\DiscreteStartupCostSum(\onOff(i,)) = \DiscreteStartupCostSum(\tildeOnOff(i,)) + \scsSumDiff{t}(\offlineNumber{t}{\onOff(i,)},\offlineNumberRight{t}{\onOff(i,)}).
	\end{equation*}
	These vectors, the induced subtrees, and the offline lengths $\offlineNumber{t}{\onOff(i,)}$, $\offlineNumberRight{t}{\onOff(i,)}$ are shown in \cref{figure:SCS_LiftUnexact}.

	\begin{figure}[bht]
	\centering\begin{tikzpicture}[x=0.6cm,y=-1.2cm,scale=\textwidthScaling]
		\LiftingTrees
		
		\node[tree] (n1)  at ( 1,2)  {1};
		\node[tree] (n2)  at ( 2,1)  {2};
		\node[tree] (n3)  at ( 3,2)  {3};
		\node[tree] (n4)  at ( 4,3)  {4};
		\node[tree] (n5)  at ( 5,0)  {5};
		\node[tree] (n6)  at ( 6,3)  {6};
		\node[tree,pin={90:$t$}, draw=marker] (n7)  at ( 7,2)  {7};
		\node[tree] (n8)  at ( 8,3)  {8};
		\node[tree] (n9)  at ( 9,1)  {9};
		\node[tree] (n10) at (10,3) {10};
		\node[tree] (n11) at (11,2) {11};
		
		\draw (n5) edge (n2) edge (n9);
		\draw (n2) edge (n1) edge (n3);
		\draw (n3) edge (n4);
		\draw (n9) edge (n7) edge (n11);
		\draw (n7) edge (n6) edge (n8);
		\draw (n11) edge (n10);
		
		\begin{pgfonlayer}{background}
			\coordinate (helper) at ($(n2)+(1.6,0)$);
			\fill[S1] \convexpath{helper,n2,n9,n7}{\convexRadiusB};
			\fill[white] \convexcycle{n9,n2}{\convexRadiusA};	%
			\fill[S2] \convexcycle{n9,n2}{\convexRadiusA};
			\node[S1text,anchor=south,rectangle,draw=none,inner sep=0pt] at ($(n7)!0.4!(n3)$) {$\oneSubtree{\onOff(i,)}$};
			\node[S2text,anchor=base] at ($(n2.base)!0.3!(n9.base)$) {$\oneSubtree{\tildeOnOff(i,)}$};
		\end{pgfonlayer}
		
		\foreach \i in {1,3,4,5,6,7,8,10,11} {
			\node[vector] (vt\i) at (\i,4.4) {$0$};
		}
		\foreach \i in {2,9} {
			\fill[S2fill] (\i,4.4) circle (1.5\sex);
			\node[vector] (vt\i) at (\i,4.4) {$1$};
		}
		\node[vector,anchor=base east] at (vt1.base west) {$\tildeOnOff(i,) = ($};
		\node[vector,anchor=base west] at (vt11.base east) {$)$};
		\foreach \i in {1,3,4,5,6,8,10,11} {
			\node[vector] (v\i) at (\i,5) {$0$};
		}
		\foreach \i in {2,7,9} {
			\fill[S1fill] (\i,5) circle (1.5\sex);
			\node[vector] (v\i) at (\i,5) {$1$};
		}
		\node[vector,anchor=base east] at (v1.base west) {$\onOff(i,) = ($};
		\node[vector,anchor=base west] at (v11.base east) {$)$};

		\draw[markerLine] (n2) -- (vt2);
		\draw[markerLine] (n7) -- (vt7);
		\draw[markerLine] (n9) -- (vt9);
		\draw[draw=marker,decorate,decoration={brace,amplitude=1.5\smm,mirror}] ($(2,0 |- n4.south)-(0,0.5mm)$) -- ($(7,0 |- n4.south)-(0,0.5mm)$) node[anchor=north,rectangle,inner sep=0pt,draw=none,fill=none,markerText,midway,yshift=-2\smm] {$\offlineNumber{7}{\onOff(i,)}$};
		\draw[draw=marker,decorate,decoration={brace,amplitude=1.5\smm,mirror}] ($(7,0 |- n4.south)-(0,0.5mm)$) -- ($(9,0 |- n4.south)-(0,0.5mm)$) node[anchor=north,rectangle,inner sep=0pt,draw=none,fill=none,markerText,midway,yshift=-2\smm] {$\offlineNumberRight{7}{\onOff(i,)}$};
	\end{tikzpicture}
	\caption[Removing the leaf from the induced subgraph \oneSubtree{\protect\onOff(i,)}$ results in the subgraph $\protect\oneSubtree{\protect\tildeOnOff(i,)}, induced by $\protect\tildeOnOff(i,)$$]{Removing the leaf $t = 7$ from the induced subgraph~$\oneSubtree{\onOff(i,)}$ results in the subgraph $\oneSubtree{\tildeOnOff(i,)}$ induced by~$\tildeOnOff(i,)$. The lengths $\offlineNumber{7}{\onOff(i,)}$ and $\offlineNumberRight{7}{\onOff(i,)}$ denote the offline lengths before and after period~7 (see \eqref{equation:SC_OfflineNumber} and \eqref{equation:SCS_OfflineNumberRight}).}
	\label{figure:SCS_LiftUnexact}
\end{figure}
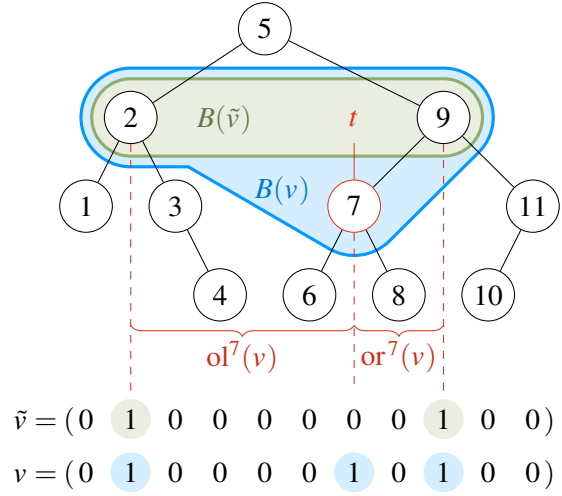
	
	By the choice of $t$, its left subtree, which contains the nodes~$t-\leftSize(t),\ldots,t-1$, does not contain any nodes from $\oneSubtree{\onOff(i,)}$. Hence, $\onOff{i}{t-\leftSize(t)} = \ldots = \onOff(i,t-1) = 0$, implying $\offlineNumber{t}{\onOff(i,)} \geq \leftSize(t)$. Since the right subtree of $t$ does not contain any nodes from $\oneSubtree{\onOff(i,)}$ either, we analogously obtain $\offlineNumberRight{t}{\onOff(i,)} \geq \rightSize(t)$.
	
	\newpage
	Thus, using the monotonicity of $\scsSumDiff{t}$, it holds that
	\begin{align*}
		\DiscreteStartupCostSum(\onOff(i,)) &= \DiscreteStartupCostSum(\tildeOnOff(i,)) + \scsSumDiff{t}(\offlineNumber{t}{\onOff(i,)},\offlineNumberRight{t}{\onOff(i,)}) \geq \DiscreteStartupCostSum(\tildeOnOff(i,)) + \scsSumDiff{t}(\leftSize(t), \rightSize(t))\\
			&\withtextl{ind.hyp.}{\geq} \sum_{\tau \in \Periods} \scsSumDiff{\tau}(\leftSize(\tau), \rightSize(\tau)) \, \tildeOnOff(i,\tau) + \scsSumDiff{t}(\leftSize(t), \rightSize(t))
			= \sum_{\tau \in \Periods} \scsSumDiff{\tau}(\leftSize(\tau), \rightSize(\tau)) \, \onOff(i,\tau).\eqqed{1em}
	\end{align*}
\end{proof}

The central argument of the proof of the last lemma is that in each step of the induction, the inequality
\begin{equation*}
	\scsSumDiff{t}(\offlineNumber{t}{\onOff(i,)},\offlineNumberRight{t}{\onOff(i,)}) \geq \scsSumDiff{t}(\leftSize(t),\rightSize(t))
\end{equation*}
holds due to $\offlineNumber{t}{\onOff(i,)} \geq \leftSize(t)$ and $\offlineNumberRight{t}{\onOff(i,)} \geq \rightSize(t)$. Assume that, for a given binary tree $B$ and a vertex of $\scsEpigraph$, this inequality is fulfilled with equality in each induction step. Then the vertex also fulfills the BTI induced by $B$ with equality. We characterize such vertices in the next lemma.
\begin{lemma}
	\label[lemma]{result:SCS_PointsOnFacet}
	For each $B \in \scsBinaryTrees$, $(\onOff(i,),\DiscreteStartupCostSum(\onOff(i,))) \in \scsVertices$, if the induced subgraph $\oneSubtree{\onOff(i,)}$ is a tree containing $\root(B)$, then $(\onOff(i,),\DiscreteStartupCostSum(\onOff(i,)))$ fulfills the BTI induced by $B$ with equality.
\end{lemma}
\begin{proof}
	Analogously to the proof of \cref{result:SCS_EpigraphInPolyhedron}, we show that $(\onOff(i,),\DiscreteStartupCostSum(\onOff(i,)))$ fulfills the binary tree inequality induced by $B \in \scsBinaryTrees$ by induction over the number of nodes~$n$ in the subgraph $\oneSubtree{\onOff(i,)}$.

	For $n = 0$, the only such vertex is $0 \in \scsVertices$, which fulfills all BTIs with equality. For $n \geq 1$, let $t$ be a leaf of $\oneSubtree{\onOff(i,)}$, and define $\tildeOnOff(i,) := \onOff(i,)-\unitVec{t}$ as in the proof of \cref{result:SCS_EpigraphInPolyhedron}. Note that since $\oneSubtree{\onOff(i,)}$ is a tree and $\root(B) \in \oneSubtree{\onOff(i,)}$, also the subgraph $\oneSubtree{\tildeOnOff(i,)}$ is a tree with $n-1$ nodes, which contains $\root(B)$ in case of $n \geq 2$.

	Since no nodes in the principal subtree of $t$, except $t$ itself, are in $\oneSubtree{\onOff(i,)}$, analogeously to the proof of \cref{result:SCS_EpigraphInPolyhedron}, it holds that $\offlineNumber{t}{\onOff(i,)} \geq \leftSize(t)$ and $\offlineNumberRight{t}{\onOff(i,)} \geq \rightSize(t)$. We show that these bounds are sharp by examining two cases for $\offlineNumber{t}{\onOff(i,)} = \leftSize(t)$ (see \cref{figure:SCS_LiftPrev}) and two cases for $\offlineNumberRight{t}{\onOff(i,)} = \rightSize(t)$ (see \cref{figure:SCS_LiftNext}).
	
	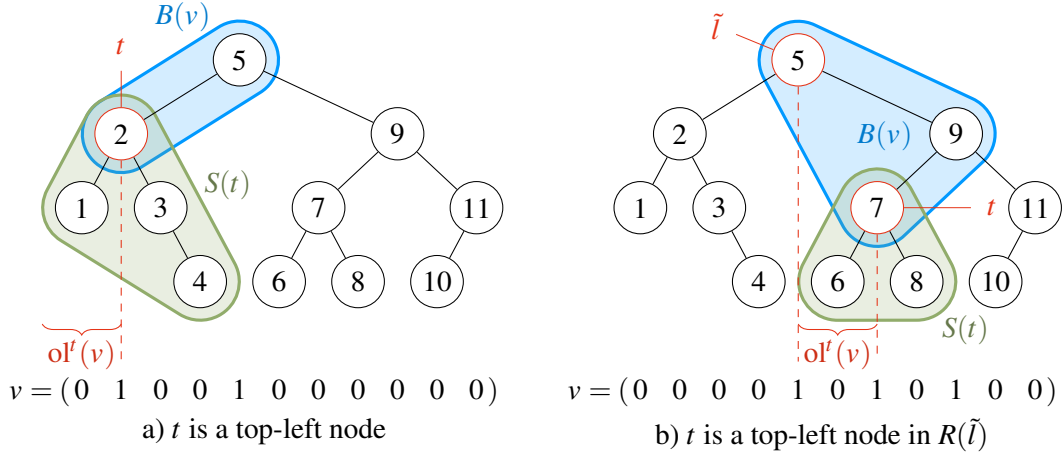
\begin{figure}[htb]
	\begin{tikzpicture}[x=0.53cm,y=-1cm,scale=\textwidthScaling]
		\LiftingTrees
	
		\node[tree] (n1)  at ( 1,2)  {1};
		\node[tree,pin={$t$}, draw=marker] (n2)  at ( 2,1)  {2};
		\node[tree] (n3)  at ( 3,2)  {3};
		\node[tree] (n4)  at ( 4,3)  {4};
		\node[tree] (n5)  at ( 5,0)  {5};
		\node[tree] (n6)  at ( 6,3)  {6};
		\node[tree] (n7)  at ( 7,2)  {7};
		\node[tree] (n8)  at ( 8,3)  {8};
		\node[tree] (n9)  at ( 9,1)  {9};
		\node[tree] (n10) at (10,3) {10};
		\node[tree] (n11) at (11,2) {11};
		
		\draw (n5) edge (n2) edge (n9);
		\draw (n2) edge (n1) edge (n3);
		\draw (n3) edge (n4);
		\draw (n9) edge (n7) edge (n11);
		\draw (n7) edge (n6) edge (n8);
		\draw (n11) edge (n10);
		
		\begin{pgfonlayer}{background}
			\fill[S1] \convexcycle{n5,n2}{\convexRadiusA};
			\fill[S2] \convexcycle{n1,n2,n4}{\convexRadiusA};
			\node[S1text,anchor=south east] at (n5.north west) {$\oneSubtree{\onOff(i,)}$};
			\node[S2text,anchor=west] at ($(n3)+(-20:0.9)$) {$\subtree(t)$};
		\end{pgfonlayer}

		\coordinate (hh) at (n4);
		\draw[draw=marker,decorate,decoration={brace,amplitude=1.5\smm,mirror}] ($(0,0 |- hh)-(0,\convexRadiusA+0.5mm)$) -- ($(2,0 |- hh)-(0,\convexRadiusA+0.5mm)$) node[anchor=north,rectangle,inner sep=0pt,draw=none,fill=none,markerText,midway,yshift=-2\smm] (or) {$\offlineNumber{t}{\onOff(,)}$};

		\foreach \i in {1,3,4,6,7,...,11} {
			\node[vector, anchor=north, yshift=-1\smm] (v\i) at (\i,0 |- or.south) {$0$};
		}
		\foreach \i in {2,5} {
			\node[vector, anchor=north, yshift=-1\smm] (v\i) at (\i,0 |- or.south) {$1$};
		}
		\node[vector,anchor=base east] at (v1.base west) {$\onOff(,) = ($};
		\node[vector,anchor=base west] at (v11.base east) {$)$};
		
		\draw[markerLine] (n2) -- (v2);
		
		\node[description, anchor=north,yshift=-1\smm] at (6,0 |- v1.south) {\llap{a) }$t$ is a top-left node};
	
	\begin{scope}[xshift=7.5cm]
	
		\node[tree] (n1)  at ( 1,2)  {1};
		\node[tree] (n2)  at ( 2,1)  {2};
		\node[tree] (n3)  at ( 3,2)  {3};
		\node[tree] (n4)  at ( 4,3)  {4};
		\node[tree,pin={165:$\tilde{l}$}, draw=marker] (n5)  at ( 5,0)  {5};
		\node[tree] (n6)  at ( 6,3)  {6};
		\node[tree,pin={[pin distance=6.3ex] 0:$t$}, draw=marker] (n7)  at ( 7,2)  {7};
		\node[tree] (n8)  at ( 8,3)  {8};
		\node[tree] (n9)  at ( 9,1)  {9};
		\node[tree] (n10) at (10,3) {10};
		\node[tree] (n11) at (11,2) {11};
		
		\draw (n5) edge (n2) edge (n9);
		\draw (n2) edge (n1) edge (n3);
		\draw (n3) edge (n4);
		\draw (n9) edge (n7) edge (n11);
		\draw (n7) edge (n6) edge (n8);
		\draw (n11) edge (n10);
		
		\begin{pgfonlayer}{background}
			\fill[S1] \convexcycle{n7,n5,n9}{\convexRadiusA};
			\fill[S2] \convexcycle{n6,n7,n8}{\convexRadiusA};
			\node[S1text,anchor=188] at ($(n5)!0.55!(n7)!0.05!(n9)$) {$\oneSubtree{\onOff(i,)}$};
		\end{pgfonlayer}
			\node[S2text,anchor=north west] at ($(n8)+(40:0.5)$) {$\subtree(t)$};
		
		\draw[draw=marker,decorate,decoration={brace,amplitude=1.5\smm,mirror}] ($(5,0 |- n6)-(0,\convexRadiusA+0.5mm)$) -- ($(7,0 |- n6)-(0,\convexRadiusA+0.5mm)$) node[anchor=north,rectangle,inner sep=0pt,draw=none,fill=none,markerText,midway,yshift=-2\smm] (or) {$\offlineNumber{t}{\onOff(,)}$};

		\foreach \i in {1,2,3,4,6,8,10,11} {
			\node[vector, anchor=north, yshift=-1\smm] (v\i) at (\i,0 |- or.south) {$0$};
		}
		\foreach \i in {5,7,9} {
			\node[vector, anchor=north, yshift=-1\smm] (v\i) at (\i,0 |- or.south) {$1$};
		}
		\node[vector,anchor=base east] at (v1.base west) {$\onOff(,) = ($};
		\node[vector,anchor=base west] at (v11.base east) {$)$};
		\draw[markerLine] (n5) -- (v5);
		\draw[markerLine] (n7) -- (v7);
		
		\node[description, anchor=north,yshift=-1\smm] at (6,0 |- v1.south) {\llap{b) }$t$ is a top-left node in $\rightSubtree(\tilde{l})$};
	\end{scope}
	\end{tikzpicture}
	\caption{Number of offline periods~$\protect\offlineNumber{t}{\protect\onOff(i,)}$ in $\protect\onOff(i,)$ before period~$t$, which is either bounded by the start of the model or by the first left ancestor~$\tilde{l}$ of node~$t$}
	\label{figure:SCS_LiftPrev}
\end{figure}
	\begin{figure}[htb]
	\centering
	\begin{tikzpicture}[x=0.53cm,y=-1cm,scale=\textwidthScaling]
		\LiftingTrees
		
		\node[tree] (n1)  at ( 1,2)  {1};
		\node[tree] (n2)  at ( 2,1)  {2};
		\node[tree] (n3)  at ( 3,2)  {3};
		\node[tree] (n4)  at ( 4,3)  {4};
		\node[tree] (n5)  at ( 5,0)  {5};
		\node[tree] (n6)  at ( 6,3)  {6};
		\node[tree] (n7)  at ( 7,2)  {7};
		\node[tree] (n8)  at ( 8,3)  {8};
		\node[tree,pin={[pin distance=2.9ex] 90:$t$}, draw=marker] (n9)  at ( 9,1)  {9};
		\node[tree] (n10) at (10,3) {10};
		\node[tree] (n11) at (11,2) {11};
	
		\draw (n5) edge (n2) edge (n9);
		\draw (n2) edge (n1) edge (n3);
		\draw (n3) edge (n4);
		\draw (n9) edge (n7) edge (n11);
		\draw (n7) edge (n6) edge (n8);
		\draw (n11) edge (n10);
		
		\begin{pgfonlayer}{background}
			\fill[S1] \convexcycle{n5,n9}{\convexRadiusA};
			\path[S2] \convexcycle{n6,n7,n9,n11,n10}{\convexRadiusA};
			\node[S1text,anchor=east] at ($(n5.west)-(0.25,0)$) {$\oneSubtree{\onOff(i,)}$};
			\node[S2text,anchor=south] at ($(n11)+(-60:0.85)$) {$\subtree(t)$};
		\end{pgfonlayer}
		
		\draw[draw=marker,decorate,decoration={brace,amplitude=1.5\smm,mirror}] ($(9,0 |- n10)-(0,\convexRadiusA+0.5mm)$) -- ($(12,0 |- n10)-(0,\convexRadiusA+0.5mm)$) node [anchor=north,rectangle,inner sep=0pt,draw=none,fill=none,markerText,midway,yshift=-2\smm] (or) {$\squeeze{\offlineNumberRight{t}{\onOff(i,)}}$};

		\foreach \i in {1,2,3,4,6,7,8,10,11} {
			\node[vector, anchor=north, yshift=-1\smm] (v\i) at (\i,0 |- or.south) {$0$};
		}
		\foreach \i in {5,9} {
			\node[vector, anchor=north, yshift=-1\smm] (v\i) at (\i,0 |- or.south) {$1$};
		}
		\node[vector,anchor=base east] at (v1.base west) {$\onOff(i,) = ($};
		\node[vector,anchor=base west] at (v11.base east) {$)$};
				
		\draw[markerLine] (n9) -- (v9);
		
		\node[description, anchor=north,yshift=-1\smm] at (6,0 |- v1.south) {\llap{a) }$t$ is a top-right node};
	
	\begin{scope}[xshift=7.5cm]
	
		\node[tree] (n1)  at ( 1,2)  {1};
		\node[tree] (n2)  at ( 2,1)  {2};
		\node[tree] (n3)  at ( 3,2)  {3};
		\node[tree] (n4)  at ( 4,3)  {4};
		\node[tree] (n5)  at ( 5,0)  {5};
		\node[tree] (n6)  at ( 6,3)  {6};
		\node[tree,pin={180:$t$}, draw=marker] (n7)  at ( 7,2)  {7};
		\node[tree] (n8)  at ( 8,3)  {8};
		\node[tree,pin={35:$\tilde{r}$}, draw=marker] (n9)  at ( 9,1)  {9};
		\node[tree] (n10) at (10,3) {10};
		\node[tree] (n11) at (11,2) {11};
		
		\draw (n5) edge (n2) edge (n9);
		\draw (n2) edge (n1) edge (n3);
		\draw (n3) edge (n4);
		\draw (n9) edge (n7) edge (n11);
		\draw (n7) edge (n6) edge (n8);
		\draw (n11) edge (n10);
		
		\begin{pgfonlayer}{background}
			\path[S1] \convexcycle{n5,n9,n7}{\convexRadiusA};
			\path[S2] \convexcycle{n7,n8,n6}{\convexRadiusA};
			\node[S1text] at ($(n5)!0.5!(n7)!0.2!(n9)$) {$\oneSubtree{\onOff(i,)}$};
			\node[S2text,anchor=north east] at ($(n6)+(-220:0.5)$) {$\subtree(t)$};
		\end{pgfonlayer}

		\draw[draw=marker,decorate,decoration={brace,amplitude=1.5\smm,mirror}] ($(7,0 |- n8)-(0,\convexRadiusA+0.5mm)$) -- ($(9,0 |- n8)-(0,\convexRadiusA+0.5mm)$) node [anchor=north,rectangle,inner sep=0pt,draw=none,fill=none,markerText,midway,yshift=-2\smm] (or) {$\squeeze{\offlineNumberRight{t}{\onOff(i,)}}$};

		\foreach \i in {1,2,3,4,6,10,11} {
			\node[vector, anchor=north, yshift=-1\smm] (v\i) at (\i,0 |- or.south) {$0$};
		}
		\foreach \i in {5,7,8,9} {
			\node[vector, anchor=north, yshift=-1\smm] (v\i) at (\i,0 |- or.south) {$1$};
		}
		\node[vector,anchor=base east] at (v1.base west) {$\onOff(i,) = ($};
		\node[vector,anchor=base west] at (v11.base east) {$)$};

		\draw[markerLine] (n9) -- (n9 |- v11.north);
		\draw[markerLine] (n7) -- (v7);
	
		\node[description, anchor=north,yshift=-1\smm] at (6,0 |- v1.south) {\llap{b) }$t$ is a top-right node in $\leftSubtree(\tilde{r})$};
	\end{scope}
	\end{tikzpicture}
	\caption{Number of offline periods~$\protect\offlineNumberRight{t}{\protect\onOff(i,)}$ in vector $\protect\onOff(i,)$ after period~$t$, which is either bounded by the end of the model or by the first right ancestor~$\tilde{r}$ of node~$t$}
	\label{figure:SCS_LiftNext}
\end{figure}
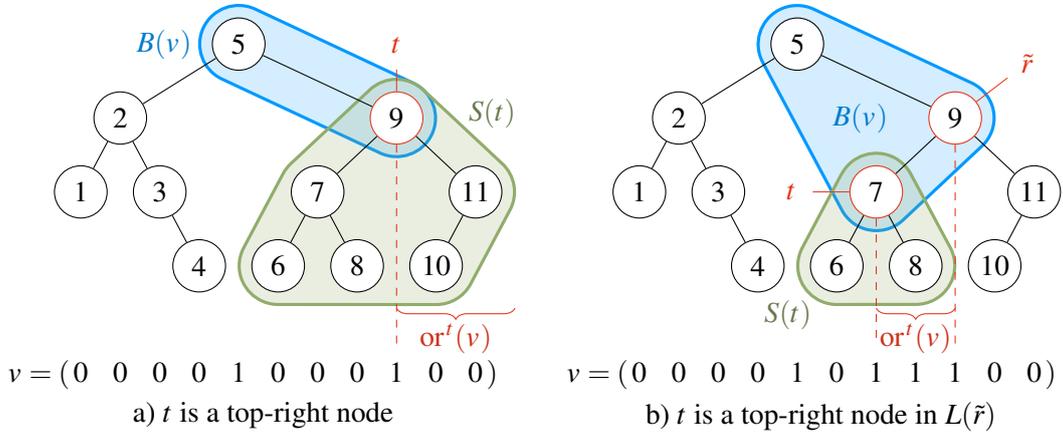

	If $t$ is a top-left node, then $\squeeze{t = \leftSize(t) + 1}$ (see \cref{result:TopLeftRanks}), implying $\squeeze{\leftSize(t) \leq \offlineNumber{t}{\onOff(i,)} \leq t - 1 = \leftSize(t)}$, and thus $\offlineNumber{t}{\onOff(i,)} = \leftSize(t)$. Otherwise, there exists a node~$\tilde{l}$ such that $t$ is a top-left node in the right subtree~$\rightSubtree(\tilde{l})$ of $\tilde{l}$. Since $\tilde{l}$ is an ancestor of $t$, we have $\tilde{l} \in \oneSubtree{\onOff(i,)}$. Hence,
	\begin{equation*}
		\leftSize(t) \leq \offlineNumber{t}{\onOff(i,)} \leq t - \tilde{l} - 1 \withtext{\cref{result:TopLeftRanks}}{=} \left(\tilde{l} + \leftSize(t) + 1 \right) - \tilde{l} - 1 = \leftSize(t).
	\end{equation*}
	
	\vskip-0.5ex
	If $t$ is a top-right node, then $t = \rightSize(t) - 1$, implying $\rightSize(t) \leq \offlineNumberRight{t}{\onOff(i,)} \leq t + 1 = \rightSize(t)$. Else, $t$ is a top-right node in the left subtree of a node $\tilde{r}$. Since $\tilde{r}$ is an ancestor of $t$, we have $\tilde{r} \in \oneSubtree{\onOff(i,)}$, and thus again
	\begin{equation*}
		\rightSize(t) \leq \offlineNumberRight{t}{\onOff(i,)} \leq \tilde{r} - t - 1 \smash{\withtext{\cref{result:TopRightSubtree}}{=}} \rightSize(t).
	\end{equation*}
	
	In conclusion, we have shown $\squeeze[0.75]{\offlineNumber{t}{\onOff(i,)} = \leftSize(t)}$ and $\squeeze[0.75]{\offlineNumberRight{t}{\onOff(i,)} = \rightSize(t)}$. Analogous to \cref{result:SCS_EpigraphInPolyhedron},
	\begin{align*}
		\DiscreteStartupCostSum(\onOff(i,)) &= \DiscreteStartupCostSum(\tildeOnOff(i,)) + \scsSumDiff{t}(\offlineNumber{t}{\onOff(i,)},\offlineNumberRight{t}{\onOff(i,)}) = \DiscreteStartupCostSum(\tildeOnOff(i,)) + \scsSumDiff{t}(\leftSize(t), \rightSize(t))\\
			&\vphantom{\Big(}\smash{\withtextl{ind.hyp.}{=}} \smash[b]{\sum_{\tau \in \Periods}} \scsSumDiff{\tau}(\leftSize(\tau), \rightSize(\tau)) \, \tildeOnOff(i,\tau) + \scsSumDiff{t}(\leftSize(t), \rightSize(t)) = \smash[b]{\sum_{\tau \in \Periods}} \scsSumDiff{\tau}(\leftSize(\tau), \rightSize(\tau)) \, \onOff(i,\tau).\eqqed{0.4em}
	\end{align*}
\end{proof}

In general, a vertex $(\onOff(i,),\DiscreteStartupCostSum(\onOff(i,)))$ may lie on a facet induced by a binary tree $B$ which does not meet the requirements of the preceding lemma, \eg for a linear start-up cost function. The next result shows that this does not happen for strictly concave start-up cost functions~$\StartupCostFunction(,)$.

\begin{lemma}
	\label[lemma]{result:SCS_PointsNotOnFacet}
	Let $\StartupCostFunction(,)$ be strictly concave, $B \in \scsBinaryTrees$, and $(\onOff(i,),\DiscreteStartupCostSum(\onOff(i,))) \in \scsVertices$. If $(\onOff(i,),\DiscreteStartupCostSum(\onOff(i,)))$ fulfills the BTI induced by $B$ with equality, then either $\onOff(i,) = 0$ or the induced subgraph $\oneSubtree{\onOff(i,)}$ is a tree containing $\root(B)$.
\end{lemma}
\begin{proof}
	Again we prove the statement by induction over the number of nodes~$n$ in the subgraph $\oneSubtree{\onOff(i,)}$. The case $n = 0$ is fulfilled trivially.
	
	For $n \geq 1$, analogously to the proof of \cref{result:SCS_EpigraphInPolyhedron}, choose a node~$t$ such that the subtree of~$t$ in $B$ does not contain any other nodes from $\oneSubtree{\onOff(i,)}$. This choice implies $\offlineNumber{t}{\onOff(i,)} \geq \leftSize(t)$ and $\offlineNumberRight{t}{\onOff(i,)} \geq \rightSize(t)$, and thus, because of concavity,
	\begin{equation*}
		\scsSumDiff{t}(\offlineNumber{t}{\onOff(i,)},\offlineNumberRight{t}{\onOff(i,)}) \geq \scsSumDiff{t}(\leftSize(t),\rightSize(t)).
	\end{equation*}
	For $\tildeOnOff(i,) := \onOff(i,) - \unitVec{t}$, \cref{result:SCS_SumUpdate} yields
	\begin{align*}
		\DiscreteStartupCostSum(\onOff(i,)) &- \scsSumDiff{t}(\offlineNumber{t}{\onOff(i,)},\offlineNumberRight{t}{\onOff(i,)})
		= \DiscreteStartupCostSum(\tildeOnOff(i,))
		\intertext{which, since $(\tildeOnOff(i,),\DiscreteStartupCostSum(\tildeOnOff(i,))) \in \scsEpigraph$, may be bounded by}
		&\geq \sum_{\tau \in \Periods} \scsSumDiff{\tau}(\leftSize(\tau),\rightSize(\tau)) \tildeOnOff(i,\tau)
		= \sum_{\tau \in \Periods} \scsSumDiff{\tau}(\leftSize(\tau),\rightSize(\tau)) \onOff(i,\tau) - \scsSumDiff{t}(\leftSize(t),\rightSize(t))
		\intertext{and, by the choice of $\onOff(i,)$ and $t$, equals}
		&= \DiscreteStartupCostSum(\onOff(i,)) - \scsSumDiff{t}(\leftSize(t),\rightSize(t))
		\geq \DiscreteStartupCostSum(\onOff(i,)) - \scsSumDiff{t}(\offlineNumber{t}{\onOff(i,)},\offlineNumberRight{t}{\onOff(i,)}).
	\end{align*}
	Therefore, the above inequality is fulfilled with equality, and we conclude
	\begin{equation*}
		\scsSumDiff{t}(\offlineNumber{t}{\onOff(i,)},\offlineNumberRight{t}{\onOff(i,)}) = \scsSumDiff{t}(\leftSize(t),\rightSize(t)) \quad\text{and}\quad \DiscreteStartupCostSum(\tildeOnOff(i,)) = \sum_{\tau \in \Periods} \scsSumDiff{\tau}(\leftSize(\tau),\rightSize(\tau)) \tildeOnOff(i,\tau).
	\end{equation*}
	
	Firstly, as $\StartupCostFunction(,)$ is assumed to be strictly concave, $\scsSumDiff{t}$ is strictly increasing in $l$ and $r$ (see \cref{result:SCS_SumUpdate}). Recalling that $\offlineNumber{t}{\onOff(i,)} \geq \leftSize(t)$ and $\offlineNumberRight{t}{\onOff(i,)} \geq \rightSize(t)$ by the choice of $t$, we infer $\offlineNumber{t}{\onOff(i,)} = \leftSize(t)$ and $\offlineNumberRight{t}{\onOff(i,)} = \rightSize(t)$.
	
	Secondly, since $\abs{\oneSubtree{\tildeOnOff(i,)}} = \abs{\oneSubtree{\onOff(i,)}} - 1$, the induction hypothesis states that either $\tildeOnOff(i,) = 0$ or $\oneSubtree{\tildeOnOff(i,)}$ is a tree containing $\root(B)$. If $\tildeOnOff(i,) = 0$, then
	\begin{equation*}
		\size(t) = \leftSize(t) + 1 + \rightSize(t) = \offlineNumber{t}{\onOff(i,)} + 1 + \offlineNumberRight{t}{\onOff(i,)} = t - 1 + 1 + (T-t) = T,
	\end{equation*}
	hence $t$ is the root of $B$ and $\oneSubtree{\onOff(i,)}$ is the single-noded tree consisting of $t = \root(B)$.
	
	Else $\tildeOnOff(i,) \neq 0$, implying that $\oneSubtree{\tildeOnOff(i,)}$ is a tree containing $\root(B)$. Therefore $t \neq \root(B)$, meaning that $t$ has a parent~$\parent(t)$. By \eqref{equation:BT_ChildRank}, if $t$ is a right child, then
	\begin{align*}
		\parent(t) &= t - \leftSize(t) - 1 = t - \offlineNumber{t}{\onOff(i,)} - 1,
		\intertext{and if $t$ is a left child, then}
		\parent(t) &= t + \rightSize(t) + 1 = t + \offlineNumberRight{t}{\onOff(i,)} + 1.
	\end{align*}
	In both cases, $\onOff{i}{\parent(t)} = 1$ by definition of $\offlineNumber{t}{\onOff(i,)}$ or $\offlineNumberRight{t}{\onOff(i,)}$ (\eqref{equation:SC_OfflineNumber} and \eqref{equation:SCS_OfflineNumberRight}), and thus $\parent(t) \in \oneSubtree{\onOff(i,)}$. Given that $\oneSubtree{\tildeOnOff(i,)}$ is a tree containing $\root(B)$, we obtain that $\oneSubtree{\onOff(i,)}$ must be a tree containing $\root(B)$ too.\manualqed
\end{proof}

The vectors $\onOff(i,)$ encountered during a lifting process gain a non-zero entry in every step, and are thus linearly independent.

\begin{theorem}
	\label[theorem]{result:SCS_Facets}
	All binary tree inequalities induce facets of $\scsEpigraph$.
\end{theorem}
\begin{proof}
	For each binary tree $B \in \scsBinaryTrees$, the induced BTI is valid for $\scsEpigraph$ (see \cref{result:SCS_EpigraphInPolyhedron}).
	
	Choose a permutation~$\sigma$ of $\Periods$ such that the nodes $\sigma(t)$ are ordered by their depth~$\depth{\sigma(t)}$ in $B$, \ie such that for each $t \in \discrange{1}{T-1}$ it holds that $\depth{\sigma(t)} \leq \depth{\sigma(t + 1)}$. For each $\idx \in \Periods$, define the vector
	\begin{equation*}
		\forallPeriods: \qquad \onOff(\idx,t) := \begin{cases}
			1 & \text{if $\sigma^{-1}(t) \leq \idx$,}\\
			0 & \text{else,}
		\end{cases}
	\end{equation*}
	which induces the subgraph $\oneSubtree{\onOff(\idx,)}$ containing the nodes $\set{\sigma(1), \ldots, \sigma(\idx)}$. $\oneSubtree{\onOff(\idx,)}$ fulfills the requirements of \cref{result:SCS_PointsOnFacet}, and thus the vertex~$(\onOff(\idx,), \DiscreteStartupCostSum(\onOff(\idx,)))$ fulfills the BTI with equality.
	
	Denoting the permutation matrix associated to $\sigma$ with $\Pi_\sigma$, we obtain
	\begin{equation*}
		\Pi_\sigma \cdot
		\begin{pmatrix}
			\onOff(1,) & \onOff(2,) & \ldots & \onOff(T,)
		\end{pmatrix} = 
		\begin{pmatrix}
			1 & \cdots & \cdots & 1\\
			0 & 1 & \cdots & 1\\
			\vdots & \ddots & \ddots & \vdots\\
			0 & \cdots & 0 & 1
		\end{pmatrix},
	\end{equation*}
	and thus $0, (\onOff(1,), \DiscreteStartupCostSum(\onOff(1,)))$, $\ldots$, $(\onOff(T,), \DiscreteStartupCostSum(\onOff(T,)))$ are affine linearly independent. Since they amount to $\squeeze{T+1}$ points on the face~$F$ of $\scsEpigraph$ induced by the BTI, $F$ must be a facet.\hbox{\manualqed}
\end{proof}

\section{Sufficiency of the BTIs}
\label{section:SCS_HRepresentation}

In this subsection, we prove that the binary tree inequalities (\cref{definition:SCS_BinaryTreeInequality}), together with the trivial facets $0 \leq \onOff(i,t) \leq 1$, are sufficient for an \hrepresentation of $\scsEpigraph$. To do so we show that the facets induced by all the BTIs fully describe the lower boundary of $\scsEpigraph$. To this end, we extend the characterization of the facets on which a vertex $(\onOff(i,),\DiscreteStartupCostSum(\onOff(i,)))$ lies (\cref{result:SCS_PointsOnFacet}) to points with $\onOff(i,) \in [0,1]^T$.

\Cref{result:SCS_PointsOnFacet} provides the sufficient (but not necessary) condition for the vertices of $\scsEpigraph$ lying on a certain facet: \enquote{A vertex $(\onOff(i,),\DiscreteStartupCostSum(\onOff(i,)))$ lies on the facet corresponding to a binary tree $B \in \scsBinaryTrees$, if the induced subgraph $\oneSubtree{\onOff(i,)}$ is a tree containing $\root(B)$.} In other words, each node~$t \in \oneSubtree{\onOff(i,)}$ needs to be connected to $\root(B)$ within $\oneSubtree{\onOff(i,)}$. The unique path $P_t$ from $t$ to $\root(B)$ consists of its ancestors,
\begin{equation*}
	P_t = t \to \parent(t) \to \parentBraces{\parent(t)} \to \parentBraces{\parentBraces{\parent(t)}} \to \ldots \to \root(B).
\end{equation*}
By definition, $t \in \oneSubtree{\onOff(i,)} \equals \onOff(i,t) = 1$. So, $\oneSubtree{\onOff(i,)}$ contains the paths $P_t$ for all $t \in \oneSubtree{\onOff(i,)}$ iff
\begin{equation*}
	\onOff(i,t) = 1 \implies \onOffBraces{i}{\parent(t)} = 1 \quad \text{for all $t \in \Periods \setminus \set{\root(B)}$},
\end{equation*}
which, since $\onOff(i,) \in \set{0,1}^T$, is equivalent to
\begin{equation}
	\label{equation:SCS_CartesianCondition}
	\onOff(i,t) \leq \onOff{i}{\parent(t)} \quad \text{for all $t \in \Periods \setminus \set{\root(B)}$}.
\end{equation}

Coincidentally, this condition is also important when searching efficiently in a point set in the Cartesian plane, and is denoted by \enquote{$B$ is a Cartesian tree for $\onOff(i,)$} in \cite{vuillemin_unifying_1980}. We use a definition adapted to our purposes, equivalent to the recursive definition of Cartesian trees in \cite{gabow_scaling_1984}.
\begin{definition}
	\label[definition]{definition:CartesianTree}
	For each $\onOff(i,) \in \reals^T$, a binary tree~$B \in \scsBinaryTrees$ with
	\begin{equation*}
		\onOff(i,t) \leq \onOff{i}{\parent(t)} \quad \text{for all $t \in \Periods \setminus \set{\root(B)}$}
	\end{equation*}
	is called a \emph{Cartesian tree} for $\onOff(i,)$.
\end{definition}

The following lemma shows that condition~(\ref{equation:SCS_CartesianCondition}) applies to any $\onOff(i,) \in [0,1]^T$ as well.
\begin{lemma}
	\label[lemma]{result:SCS_FracPointsOnFacet}
	If $\onOff(i,) \in [0,1]^T$ and $B \in \scsBinaryTrees$ is a Cartesian tree for $\onOff(i,)$, then $(\onOff(i,),\StartupCostSum(\onOff(i,)))$ lies on the facet of $\scsEpigraph$ induced by $B$.
\end{lemma}
\begin{proof}
For each $\onOff(i,) \in [0,1]^T$ and each Cartesian tree $B \in \scsBinaryTrees$ for $\onOff(i,)$, we decompose the vector~$\onOff(i,)$ into binary vectors $\tildeOnOff(\idx,)$, which, paired with the summed start-up costs~$\DiscreteStartupCostSum(\tildeOnOff(\idx,))$, lie on the facet of $\scsEpigraph$ induced by $B$. This decomposition is depicted schematically in \cref{figure:SCS_Decomposition}.

\begin{figure}[!h]
\centering
\begin{tikzpicture}[x=1mm,y=1mm,scale=\textwidthScaling]
	\decompositionDefinitions

	\decompositionAxis{\bigw}{\bigh}

	\begin{scope}[xscale=\bigw, yscale=\bigh]
		\node (original) at (2.5,1.2) {fractional vector $\onOff(i,) \in [0,1]^5$};
		\foreach \t in {1,2,3,4,5} {
			\foreach \i in {5,4,3,2,1}{
				\pgfmathparse{int(5-\i+1)}
				\path[part\pgfmathresult] ({\t-1},{\onOffV[\sigmaV[\i+1]]}) rectangle (\t,\onOffV[\sigmaV[\i]]);
				\pgfmathparse{\sigmaV[\i]}
				\ifnum\t=\pgfmathresult
					\breakforeach
				\fi
			}
			\node[anchor=north] at ({\t-0.5},0) {$\pgfmathparse{\sigmaInfV[\t]}\sigma(\pgfmathprintnumber[int trunc]{\pgfmathresult})$};
		}
	\end{scope}

	\newcommand{\smallPic}[4]{
		\begin{scope}[yshift={-47mm+#1*14.5mm}]
			\decompositionAxis*{\smallw}{\smallh}
			\begin{scope}[xscale=\smallw, yscale=\smallh]
				\ifthenelse{\not\isempty{#3}}{
					\node[anchor=east] (p#1) at (-0.2,0.5) {$\onOff{i}{#2}-\onOff{i}{#3} \ \cdot$};
				}{
					\node[anchor=east] (p#1) at (-0.2,0.5) {$\onOff{i}{#2} \ \cdot$};
				}
				\foreach \k in {#4}{
					\path[part#1] (\k-1,0) rectangle (\k,1);
					\node at (\k-0.5,0.5) {1};
				}
			\end{scope}
		\end{scope}
	}
	
	\begin{scope}[xshift=\bigw*5mm+49mm, yshift=\bigh*0.5mm]
		\node (vectorsText) at (2.5*\smallw,43) {vectors};
		\draw [decorate,decoration={brace,amplitude=2mm}]
($(0,0 |- vectorsText.south)-(0,3)$) -- ($(5*\smallw,0 |- vectorsText.south)-(0,3)$);
		\node[anchor=east] (factorsText) at (-0.3*\smallw,0 |- vectorsText) {factors};
		\draw [decorate,decoration={brace,amplitude=2mm,aspect=0.65}]
($(-3*\smallw,0 |- vectorsText.south)-(0,3)$) -- ($(-0.7*\smallw,0 |- vectorsText.south)-(0,3)$);
		\node[anchor=base] at ($(vectorsText.base west)!0.5!(factorsText.base east)$) {and};
		\node[anchor=base east] (convexCombinationText) at (factorsText.base west) {convex combination:};
		
		\smallPic{5}{5}{1}{5}
		\smallPic{4}{1}{3}{1,5}
		\smallPic{3}{3}{4}{1,3,5}
		\smallPic{2}{4}{2}{1,3,4,5}
		\smallPic{1}{2}{}{1,2,3,4,5}
	\end{scope}
	
	\begin{scope}[xscale=\bigw, yscale=\bigh]
		\tikzstyle{conn}=[-latex, shorten >=3, shorten <=5, looseness=1.8];
		\draw[conn] (5,{\onOffV[\sigmaInfV[1]]/2}) -- ++(0.4,0) to[out=0, in=180] (p1);
		\foreach \i in {2,3,4,5} {
			\draw[conn, out=0, in=180] (5,{\onOffV[\sigmaV[5-\i+1]]/2+\onOffV[\sigmaV[5-\i+2]]/2}) to (p\i);
		}
	\end{scope}
\end{tikzpicture}
	\caption[Decomposition of a fractional vector into vectors known to lie on a single facet (with appropriate summed start-up costs).]{Decomposition of a fractional vector $\onOff(i,) \in [0,1]^5$ into binary vectors $\onOff(j,) \in \set{0,1}^5$, such that all points $(\onOff(j,),\StartupCostSum(\onOff(j,)))$ lie on the same facet.}
	\label{figure:SCS_Decomposition}
\end{figure}
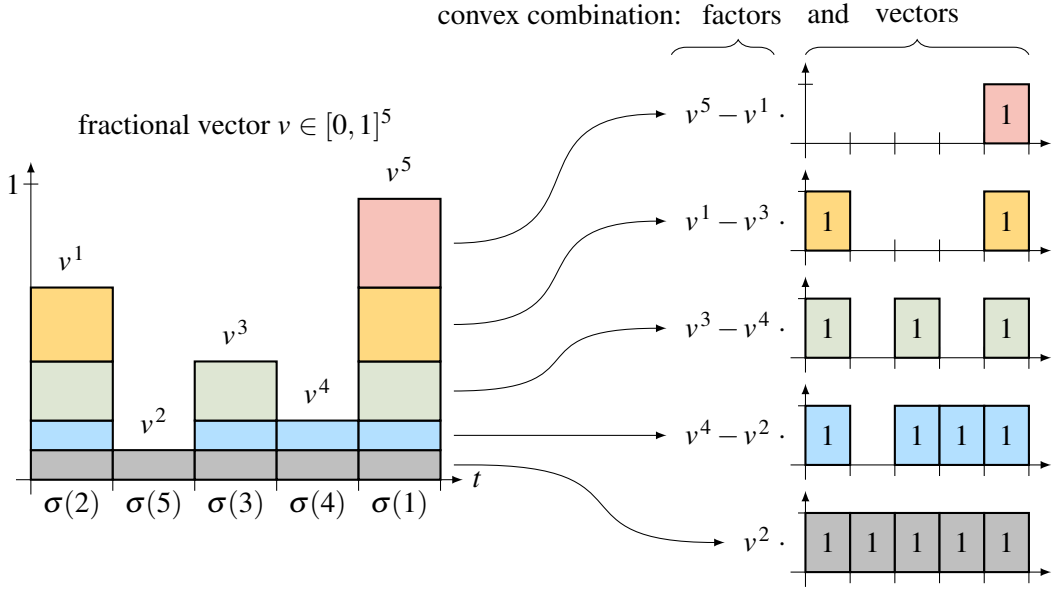
	
Choose a permutation $\sigma$ which orders the indices $t \in \Periods$ by decreasing value of~$\onOff(i,t)$, and increasing depth $\depth(t)$ in $B$ if the values $\onOff(i,t)$ are equal, \ie such that
	\begin{align*}
		\forall\,\idx_1, \idx_2 \in \Periods, \idx_1 < \idx_2: \qquad \onOff{i}{\sigma(\idx_1)} &\geq \onOff{i}{\sigma(\idx_2)} \quad\text{and}\\
		\onOff{i}{\sigma(\idx_1)} &= \onOff{i}{\sigma(\idx_2)} \quad \implies \quad \depth(\sigma(\idx_1)) \leq \depth(\sigma(\idx_2)).
	\end{align*}

	\noindent For all $\idx \in \Periods$, define the vectors $\tildeOnOff(\idx,) \in \set{0,1}^T$ and the coefficients $\lambda_\idx$ as
	\begin{equation*}
		\tildeOnOff(\idx,t) := \begin{cases}
			1 & \text{if } \sigma^{-1}(t) \leq \idx,\\
			0 & \text{else,}
		\end{cases}
		\qquad\text{and}\qquad
		\lambda_\idx := \begin{cases}
			\onOff{i}{\sigma(\idx)} - \onOff{i}{\sigma(\idx+1)} & \text{for } \idx \leq T-1,\\
			\onOff{i}{\sigma(T)} & \text{else.}
		\end{cases}
	\end{equation*}
	By definition of $\sigma$, all $\lambda_\idx$ are non-negative, and obviously they sum up to $\onOff{i}{\sigma(1)}$. Hence, for each $t \in \Periods$, it holds that
\begin{equation*}
	\sum_{\idx = 1}^T \lambda_\idx \tildeOnOff(\idx,t) =
	\sum_{\idx = \sigma^{-1}(t)}^T \lambda_\idx =
	\sum_{\idx = \sigma^{-1}(t)}^{T-1} (\onOff{i}{\sigma(\idx)} - \onOff{i}{\sigma(\idx+1)}) + \onOff{i}{\sigma(T)} =
	\onOff{i}{\sigma(\sigma^{-1}(t))} = \onOff(i,t).
\end{equation*}

\noindent Therefore, the vector~$\onOff(i,)$ indeed is a convex combination of the vectors $\tildeOnOff(\idx,)$ and $0$,
	\begin{equation*}
		\onOff(i,) = \sum_{\idx = 1}^T \lambda_\idx \tildeOnOff(\idx,) + (1 - \onOff{i}{\sigma(1)}) \:0.
	\end{equation*}

	We proceed by showing that the vertices $(\tildeOnOff(\idx,), \DiscreteStartupCostSum(\tildeOnOff(\idx,)))$ lie on the facet induced by $B$. For each $\idx \in \Periods$, the induced subgraph $\oneSubtree{\tildeOnOff(\idx,)}$ is the induced subgraph of $B$ on nodes $\set*{t \in \Periods \given \sigma^{-1}(t) \leq \idx}$. Consider an arbitrary node~$t \in \oneSubtree{\tildeOnOff(\idx,)}$ and its parent $\parent(t)$. Due to the definition of $B$, both
	\begin{equation*}
		\onOff(i,t) \leq \onOff{i}{\parent(t)} \quad\text{and}\quad \depth(t) > \depth(\parent(t)) \quad \text{hold.}
	\end{equation*}
	So, by choice of the permutation~$\sigma$, we have
	\begin{equation*}
		\sigma^{-1}(\parent(t)) < \sigma^{-1}(t) \leq \idx,
	\end{equation*}
	and hence $\parent(t) \in \oneSubtree{\tildeOnOff(\idx,)}$ too. As noted in the motivation preceding this lemma, this condition is equivalent to \enquote{$\oneSubtree{\tildeOnOff(\idx,)}$ contains the root of $B$ and is a tree}, and thus \cref{result:SCS_PointsOnFacet} implies that $(\tildeOnOff(\idx,), \DiscreteStartupCostSum(\tildeOnOff(\idx,)))$ lies on the facet induced by $B$.

	Being a convex combination of such points, the point $(\onOff(i,), \sum_{\idx = 1}^T \lambda_\idx \DiscreteStartupCostSum(\tildeOnOff(\idx,)))$ lies on the same facet, implying $\StartupCostSum(\onOff(i,)) = \sum_{\idx = 1}^T \lambda_\idx \DiscreteStartupCostSum(\tildeOnOff(\idx,))$ and therefore that $(\onOff(i,),\StartupCostSum(\onOff(i,)))$ lies on the facet of $\scsEpigraph$ induced by $B$.\manualqed
\end{proof}

Note that \cref{result:SCS_FracPointsOnFacet} is a generalization of \cref{result:SCS_PointsOnFacet}, since for discrete vectors $\squeeze{\onOff(i,) \in \set{0,1}^T}$ the induced subgraph $\oneSubtree{\onOff(i,)}$ is a Cartesian tree for~$\onOff(i,)$ iff it is empty or a tree containing $\root(B)$.

By the last result, each BTI defines a part of the lower boundary of $\scsEpigraph$, and thereby also a part of $\StartupCostSum$.
\begin{corollary}
	\label[corollary]{result:SCS_ExplicitFunction}
	If $B \in \scsBinaryTrees$ is a Cartesian tree for $\onOff(i,) \in [0,1]^T$, then
	\begin{equation*}
		\StartupCostSum(\onOff(i,)) = \sum_{t \in \Periods} \scsSumDiff{t}(\leftSize(t),\rightSize(t)) \, \onOff(i,t).
	\end{equation*}
\end{corollary}

By definition, a Cartesian tree for some~$\onOff(i,)$ is also a Cartesian tree for~$\lambda\onOff(i,)$ with $\lambda \geq 0$, which means \cref{result:SCS_ExplicitFunction} states that $\StartupCostSum$ is homogeneous.
\begin{corollary}
	\label[corollary]{result:SCS_Homogeneous}
	$\StartupCostSum$ is homogeneous, \ie
	\begin{equation*}
		\forall\, \lambda \in [0,1], \onOff(i,) \in [0,1]^T: \qquad \lambda\StartupCostSum(\onOff(i,)) = \StartupCostSum(\lambda \onOff(i,)).
	\end{equation*}
\end{corollary}

Since epigraphs are characterized by their lower boundary, this is equivalent to:
\begin{corollary}
	\label[corollary]{result:SCS_Separation}
	If $(\onOff(i,),\startupCostSum) \in [0,1]^T \times \reals$ and $B \in \scsBinaryTrees$ is a Cartesian tree for $\onOff(i,)$, then $(\onOff(i,), \startupCostSum)$ lies in $\scsEpigraph$ iff it fulfills the BTI induced by $B$.
\end{corollary}
\begin{proof}
\mbox{}\vskip-2.3\baselineskip
\begin{align*}
	(\onOff(i,),\startupCostSum) \in \scsEpigraph &\equals \startupCostSum \geq \StartupCostSum(\onOff(i,)) = \sum_{t \in \Periods} \scsSumDiff{t}(\leftSize(t),\rightSize(t)) \onOff(i,t)\\
	&\equals (\onOff(i,),\startupCostSum) \text{ fulfills the BTI induced by $B$}.\eqqed{3.9em}
\end{align*}
\end{proof}

Finally, since there exists a rank-labeled Cartesian tree for each $\onOff(i,) \in [0,1]^T$ (\cite{gabow_scaling_1984}), the BTIs and the trivial inequalities $0 \leq \onOff(i,t) \leq 1$ completely describe $\scsEpigraph$.
\begin{theorem}
	\label[theorem]{result:SCS_HRepresentation}
	\begin{equation}
		\label{equation:SCS_HRepresentation}
		\scsEpigraph = \set*{
			\begin{gathered}
				(\onOff(i,),\startupCostSum) \in \reals^{T+1} \subjectTo \hspace*{\fill}\\
				\qquad\begin{alignedat}{2}
					\startupCostSum &\geq \sum_{t \in \Periods} \scsSumDiff{t}( \leftSize(t), \rightSize(t)) \, \onOff(i,t), &\qquad& B \in \scsBinaryTrees\\
				0 & \leq \onOff(i,t) \leq 1, && t \in \Periods
				\end{alignedat}\
			\end{gathered}}.
	\end{equation}
\end{theorem}

It remains to discuss whether this \hrepresentation is irredundant. As noted in the paragraph before \cref{result:SCS_PointsNotOnFacet}, if the start-up cost function~$\StartupCostFunction(i,)$ is not strictly concave, \eg if $\StartupCostFunction(i,)$ is linear, then a vertex~$(\onOff(i,),\DiscreteStartupCostSum(\onOff(i,)))$ may fulfill a BTI induced by a binary tree~$B \in \scsBinaryTrees$ which is not a Cartesian tree for~$\onOff(i,)$. This leads to multiple binary trees inducing the same facet of $\scsEpigraph$, rendering the above \hrepresentation redundant.

However, if $\StartupCostFunction(i,)$ is strictly concave, \cref{result:SCS_PointsNotOnFacet} describes all vertices~$(\onOff(i,),\DiscreteStartupCostSum(\onOff(i,)))$ on a facet. By showing that different binary trees induce facets with different vertices, we prove that the given \hrepresentation is irredundant.
\begin{theorem}
	\label[theorem]{result:SCS_StrictlyConcave_Irredundant}
	If $\StartupCostFunction(i,)$ is strictly concave, then the \hrepresentation~\eqref{equation:SCS_HRepresentation} of $\scsEpigraph$ is irredundant.
\end{theorem}
\begin{proof}
	For each $B_1,B_2 \in \scsBinaryTrees$ with $B_1 \neq B_2$, we construct a vertex~$(\onOff(,),\DiscreteStartupCostSum(\onOff(,))) \in \scsVertices$ such that $B_1$ is a Cartesian tree for $\onOff(,)$, but $B_2$ is not a Cartesian tree for $\onOff(,)$. Then, by \cref{result:SCS_PointsOnFacet}, $(\onOff(,),\DiscreteStartupCostSum(\onOff(,)))$ lies on the facet induced by $B_1$, and by \cref{result:SCS_PointsNotOnFacet}, $(\onOff(,),\DiscreteStartupCostSum(\onOff(,)))$ does not lie on the facet induced by $B_2$, proving that the induced facets are not equal.
	
	If $\root(B_1) \neq \root(B_2)$, then $(\unitVec{\root(B_1)},\DiscreteStartupCostSum(\unitVec{\root(B_1)}))$ clearly is such a vertex. Otherwise, we have $r := \root(B_1) = \root(B_2)$, and therefore the edge sets~$E_1$ of $B_1$ and $E_2$ of $B_2$ differ. Each edge in $B_1$ connects a node~$t$ to its parent~$\parentT(B_1,t)$. Choose $e = \set{t,\parentT(B_1,t)} \in E_1 \setminus E_2$ with minimal~$\depth(t)$. Since $e \notin E_2$ and $t \neq r$, $t$ has different parents $p_1 := \parentT(B_1,t)$ and $p_2 := \parentT(B_2,t)$ in $B_1$ and $B_2$. Now, denote the path from $r$ to $p_1$ in $B_1$ by $P$, and define $\onOff(,) \in \set{0,1}^T$ as
	\begin{equation*}
		\forall\, \tau \in \Periods: \qquad \onOff(,\tau) := \begin{cases}
			1 & \text{if $\tau = t$ or $\tau \in P$,}\\
			0 & \text{else.}
		\end{cases}
	\end{equation*}
	By definition, $B_1$ is a Cartesian tree for $\onOff(,)$.
	
	We conclude this proof by showing that $p_2 \notin P$, which implies $\onOff(,p_2) = 0$, and thereby that $B_2$ is not a Cartesian tree for $\onOff(,)$: Assume $p_2 \in P$. Then, since choosing $e$ such that $\depth(t)$ is minimal, we have $P \subseteq B_2$, and therefore $p_1$ is a descendant of $p_2$ in both $B_1$ and $B_2$. Since $t$ is a child of $p_2$ in $B_2$ and $t \notin P$, it holds that $t$ and $p_1$ are separated to the subtrees~$\leftSubtreeT(B_2,p_2)$ and $\rightSubtreeT(B_2,p_2)$ of $p_2$, and thus either $t < p_2 < p_1$ or $p_1 < p_2 < t$. However, $t$ is a child of $p_1$ in $B_1$, and thus both, $p_1$ and $t$, either lie in the left subtree~$\leftSubtreeT(B_1,p_2)$ or right subtree~$\rightSubtreeT(B_1,p_2)$ of $p_2$ in $B_1$. Since this implies either $t,p_1 < p_2$ or $p_2 < t,p_1$, we obtain a contradiction, proving $p_2 \notin P$.\manualqed
\end{proof}

By \cite{rosen_handbook_1999}, the number of different binary trees on $T$ nodes is the $T$-th Catalan number, which allows us to count the number of facets of $\scsEpigraph$.
\begin{corollary}
	\label[corollary]{result:SCS_CatalanManyFacets}
	If $\StartupCostFunction(i,)$ is strictly concave, the number of facets of $\scsEpigraph$ is
	\begin{equation*}
		C_T + 2T \sim \frac{4^T}{T^{\frac{3}{2}} \sqrt{\pi}},
	\end{equation*}
	where $C_T$ denotes the $T$-th Catalan number.
\end{corollary}

\section{Separation}
\label{section:SCS_Separation}

Generally, the \hrepresentation of $\scsEpigraph$ given in the last subsection is of exponential size, and thus not (directly) suitable for computational purposes. This is overcome by a cutting plane approach based on an exact separation algorithm for $\scsEpigraph$ presented in this section. Assuming that the start-up cost function~$\StartupCostFunction{i}{}$ can be evaluated in $\bigO(1)$, as is the case for the exponential start-up cost function \eqref{equation:SC_FunctionForTime}, we show that this separation algorithm has a running time of $\bigO(T)$. Otherwise, the running time would change proportionally.

\Cref{result:SCS_FracPointsOnFacet} states that a point $(\onOff(i,),\startupCostSum) \in [0,1]^T \times \reals$ lies in $\scsEpigraph$ iff the BTI induced by the Cartesian tree for~$\onOff(i,)$ is fulfilled. Thus, the idea of the separation algorithm for $\scsEpigraph$ is to find a Cartesian tree for the vector~$\onOff(i,)$, and construct its induced BTI.

A linear-time algorithm for finding a Cartesian tree for $\onOff(i,)$ has already been given in \cite{gabow_scaling_1984}, and is denoted as \emph{FindCartesianTree}. In summary, this algorithm starts with a tree with a single node $1$, and iteratively adds the remaining nodes~$t \in \discrange{2}{T}$. The key observation is that the node~$t$ in each iteration must be added such that it results to be
\begin{itemize*}
  \item the last top-right node (as to receive the correct rank~$t$), and to be
  \item beneath all top-right nodes~$\tau$ with $\onOff{i}{\tau} > \onOff(i,t)$ and above all top-right nodes~$\tau$ with $\onOff{i}{\tau} < \onOff(i,t)$.
\end{itemize*}

The algorithm represents the resulting Cartesian tree by the left and right childs~$\leftChild(t)$ and $\rightChild(t)$ of each node~$t \in \Periods$. To construct the induced BTI, its coefficients
\begin{equation}
	\label{equation:SCS_CoefficientsForAlg}
	\bticoeff(t) = \scsSumDiff{t}(\leftSize(t),\rightSize(t)) = \StartupCostFunction{i}{\offlineLengthTree{\leftSubtree(t)}} + \StartupCostFunction{i}{\offlineLengthTree{\rightSubtree(t)}} - \StartupCostFunction{i}{\offlineLengthTree{\subtree(t)}}
\end{equation}
need to be computed. This requires the subtree sizes~$\leftSize(t)$, $\rightSize(t)$ and $\size(t)$, which due to \eqref{equation:BT_SubtreeRanks} and \eqref{equation:BT_ChildRank} equal
\begin{align}
	\label{equation:SCS_LeftSizeAlg}
	\leftSize(t) &= \size{\leftChild(t)} = \leftSize{\leftChild(t)} + 1 + \rightSize{\leftChild(t)} = \leftSize{\leftChild(t)} + t - \leftChild(t),\\
	\label{equation:SCS_RightSizeAlg}
	\rightSize(t) &= \size{\rightChild(t)} = \leftSize{\rightChild(t)} + 1 + \rightSize{\rightChild(t)} = \rightChild(t) - t + \rightSize{\rightChild(t)},\\
	\label{equation:SCS_SizeAlg}
	\size(t) &= \leftSize(t) + 1 + \rightSize(t).
\end{align}

\noindent By the definition of the offline lengths $\offlineLength(t,l)$ (see \eqref{equation:SC_OfflineLength}), we have
\begin{equation*}
	\forall\ t \in \Periods, l_1 \in \discrange*{0}{t-1}, l_2 \in \discrange*{1}{t-l_1-1}: \quad \offlineLength(t,l_1 + l_2) = \offlineLength(t,l_1) + \offlineLength(t-l_1,l_2),
\end{equation*}
and the desired offline lengths $\offlineLengthTree{\leftSubtree(t)}$, $\offlineLengthTree{\rightSubtree(t)}$ and $\offlineLengthTree{\subtree(t)}$ may be derived from \eqref{equation:SC_OfflineLengthTree},\eqref{equation:SC_OfflineLengthTreeLeftRight} for each $t \in \Periods$ as
\begin{align}
	\label{equation:SCS_OLTLeft_Formula}
	\offlineLengthTree{\leftSubtree(t)} &= \offlineLength{t}{\leftSize(t)} = \offlineLength(t,t-1) - \begin{cases}
			\offlineLength{t - \leftSize(t)}{t - \leftSize(t)-1} & \text{if $\leftSize(t) < t-1$},\\
			0 & \text{else.}
	\end{cases}\\
	\label{equation:SCS_OLTRight_Formula}
	\offlineLengthTree{\rightSubtree(t)} &= \offlineLength{t+\rightSize(t)+1}{\rightSize(t)} = \offlineLength{t + \rightSize(t)+1}{t+\rightSize(t)} - \offlineLength{t+1}{t},\\
	\label{equation:SCS_OLTPrincipal_Formula}
	\offlineLengthTree{\subtree(t)} &= \offlineLength{t+\rightSize(t)+1}{\size(t)} = \offlineLengthTree{\leftSubtree(t)} + \offlineLengthTree{\rightSubtree(t)} + \periodLength(t).
\end{align}

\noindent Using \eqref{equation:SCS_CoefficientsForAlg}-\eqref{equation:SCS_OLTPrincipal_Formula}, the coefficients $\bticoeff(t)$ of the induced BTI can be computed in linear time.

\newpage
\begin{proposition}
	\label[proposition]{result:SCS_SeparationAlgorithm}
	Algorithm~\ref{algorithm:SCS_Separation} solves the separation problem for $\scsEpigraph$ in $\bigO(T)$.
\end{proposition}

\begin{algorithm}[!ht]
	\caption{SeparateBTI}
	\label{algorithm:SCS_Separation}
	\SetKwInOut{Input}{Input}
	\SetKwInOut{Output}{Output}
	\SetKwFunction{FindCartesianTree}{FindCartesianTree}

	\Input{Point $(\onOff(i,),\startupCostSum) \in [0,1]^T \times \reals$}
	\Output{$(\onOff(i,),\startupCostSum) \in \scsEpigraph$, or a separating inequality.}

	\vskip2.5ex
	$B \gets$ FindCartesianTree($\onOff(i,)$)\;

	\vskip2.5ex
	\For{$t = 1, \ldots, T$}{
		$\leftSize(t) := \begin{cases}
			\leftSize{\leftChild(t)} + t - \leftChild(t) & \text{if $\leftChild(t) \neq \emptyset$}\\
			0 & \text{else\;}
		\end{cases}$
	}
	\For{$t = T, \ldots, 1$}{
		$\rightSize(t) := \begin{cases}
			\rightChild(t) - t + \rightSize{\rightChild(t)} & \text{if $\rightChild(t) \neq \emptyset$}\\
			0 & \text{else\;}
		\end{cases}$\\
		$\size(t) := \rightSize(t) + \leftSize(t) + 1$\;
	}

	\vskip2.5ex
	$\offlineLength(1,0) := \preOffline(i)$\;
	\For{$t = 2,\ldots,T$}{
		$\offlineLength(t,t-1) := \offlineLength(t-1,t-2) + \periodLength(t-1)$\;
	}
	\For{$t = 1,\ldots,T$}{
		$\offlineLengthTree{\leftSubtree(t)} := \offlineLength(t,t-1) - \begin{cases}
			\offlineLength{t - \leftSize(t)}{t - \leftSize(t)-1} & \text{if $\leftSize(t) < t-1$},\\
			0 & \text{else.}\;
		\end{cases}$\\
		$\offlineLengthTree{\rightSubtree(t)} := \offlineLength(t + \rightSize(t)+1,t + \rightSize(t)) - \offlineLength{t+1}{t}$\;
		$\offlineLengthTree{\subtree(t)} := \offlineLengthTree{\leftSubtree(t)} + \periodLength(t) + \offlineLengthTree{\rightSubtree(t)}$\;
	}

	\vskip2.5ex
	\For{$t \in \Periods$}{
		$\bticoeff(t) := \begin{cases}
			\StartupCostFunction{i}{\offlineLengthTree{\leftSubtree(t)}} + \StartupCostFunction{i}{\offlineLengthTree{\rightSubtree(t)}} - \StartupCostFunction{i}{\offlineLengthTree{\subtree(t)}} & \text{if $t + \rightSize(t) < T$,}\\
			\StartupCostFunction{i}{\offlineLengthTree{\leftSubtree(t)}} & \text{else.}\;
		\end{cases}$
	}
	\If{$\ds\startupCostSum \geq \sum_{t \in \Periods} \bticoeff(t) \onOff(i,t)$}{
		\Return $(\onOff(i,),\startupCostSum) \in \scsEpigraph$\;
	} \Else {
		\Return $(\onOff(i,),\startupCostSum)$ may be separated from $\scsEpigraph$ by the BTI with coefficients~$\bticoeff(t)$\;
	}
\end{algorithm}

\newpage
\section{Experiments}
\label{section:Experiments}

\def\defaultCompany{}

This section mainly compares the integrality gap of a Unit Commitment problem which models $\scsEpigraph$ using binary tree inequalities (\emph{BTI}) to a Unit Commitment problem using the start-up cost formulations presented in \cite{carrion_computationally_2006} (\emph{1-Bin}), \cite{simoglou_optimal_2010} (\emph{3-Bin}), and \cite{silbernagl_improving_2014} (\emph{1-Bin*} and \emph{Temp}).

\subsection{Scenario}
\label{section:EX_Parameters}

We model the German electricity system in 2025, using the same scenario as \cite{silbernagl_improving_2014}. For each unit~$i \in \Units$, we consider the following parameters:
\begin{itemize*}
  \item the variable and fixed production cost coefficients $A_i$ and $B_i$,
  \item the minimal and maximal production $\minProd(i)$ and $\maxProd(i)$,
  \item the maximal upward ramping speed while online ($\rampup(i)$) and at start-up ($\startupRamp(i)$), and
  \item the maximal downward ramping speed while online ($\rampdown(i)$) and at shutdown ($\shutdownRamp(i)$).
\end{itemize*}
The values of these parameters are based on the the German power system (\cite{german_federal_network_agency_list_2014}) and complemented by assumptions on minimal production, efficiency, and start-up costs, which are partly derived from \cite{kumar_power_2012}, \cite{eurelectric_efficiency_2003}, and \cite{egerer_electricity_2014}. As is expected for 2025, all nuclear units are replaced by combined cycle plants, resulting in a total of 223 thermal units.

The demand data is taken from \cite{entsoe_hourly_2007} and scaled to a yearly consumption of 520~TWh. As wind, solar, biomass, and hydro plants are not modeled explicitly, their expected production is subtracted from this demand, yielding the residual demand~$d^t$. The production from wind and solar plants is computed based on \cite{rienecker_merra:_2011} in combination with the expected installed capacity in 2025 (50~GW wind capacity, 50~GW solar capacity). Biomass and hydro plants are assumed to produce at full capacity (5.5~GW biomass capacity, 4.5~GW hydro capacity).

\subsection{Base Unit Commitment Formulation}

As the focus of the experiments lies on comparing start-up cost formulations, we restrict ourselves to a basic model for the Unit Commitment problem. Its objective is to fulfill the residual demand~$d^t$ in each period~$t$ \eqref{equation:EX_Demand}, while minimizing production costs~$\prodCost(i,t)$ and start-up costs~$\startupCostSum$ \eqref{equation:EX_Objective}. The production costs~$\prodCost(i,t)$ of unit~$i$ in period~$t$ are modeled as a linear function of its electricity output~$\prod(i,t)$ and its operational state~$\onOff(i,t)$ in \eqref{equation:EX_ProdCosts}. The production limits in \eqref{equation:EX_ProdLimits} and the ramping limits in \eqref{equation:EX_RampUp}-\eqref{equation:EX_Shutdown} are modeled as in \cite{carrion_computationally_2006}.

\begin{alignat}{2}
	\label{equation:EX_Objective}
	& \min \smash[b]{\sum_{i \in \Units\mkern3mu} \sum_{t \in \Periods} \prodCost(i,t) + \sum_{i \in \Units} \startupCostSum}&&\\[1.5ex]
	\label{equation:EX_Demand}
	\textstyle\sum_{i \in \Units} \prod(i,t) &= d^t &\qquad& \forallPeriods,\\
	\label{equation:EX_ProdCosts}
	\prodCost(i,t) &= A_i \prod(i,t) + B_i \onOff(i,t) && \forallUnitsPeriods\\
	\label{equation:EX_ProdLimits}
	\minProd(i) \onOff(i,t) &\leq \prod(i,t) \leq \maxProd(i) \onOff(i,t) && \forall\,i \in \Units, t \in \discrange{1}{T},\\
	\label{equation:EX_RampUp}
	\prod(i,t) &\leq \prod(i,t-1) + \rampup(i) \onOff(i,t-1) + \startupRamp(i) (\onOff(i,t) - \onOff(i,t-1)) + \maxProd(i)(1-\onOff(i,t)) && \forall\,i \in \Units, t \in \discrange{2}{T},\\
	\label{equation:EX_RampDown}
	\prod(i,t) &\geq \prod(i,t-1) - \rampdown(i) \onOff(i,t) - \shutdownRamp(i) (\onOff(i,t-1) - \onOff(i,t)) + \maxProd(i)(1-\onOff(i,t-1)) && \forall\,i \in \Units, t \in \discrange{2}{T},\\
	\label{equation:EX_Shutdown}
	\prod(i,t) &\leq \maxProd(i) \onOff(i,t+1) + \shutdownRamp(i) (\onOff(i,t) - \onOff(i,t+1)) && \forall\,i \in \Units,t \in \discrange*{1}{T-1},\\
	\label{equation:EX_OperationalState}
	\onOff(i,t) &\in \set{0,1},	\prod(i,t), \prodCost(i,t) \in \nnreals && \forallUnitsPeriods.
\end{alignat}

\subsection{Start-up Cost Formulations}

The term~$\startupCostSum$ in \eqref{equation:EX_Objective} 
represents the summed start-up costs of unit~$i$ according to the thermal start-up cost function in \eqref{equation:SC_ThermalStartupCosts},
\begin{equation*}
	\forallUnits, L \in \nnreals: \qquad \StartupCostFunction(i,L) = \smash[t]{\begin{cases}
		\heatingCost(i) (1 - \exp{\heatloss(i) L}) + \fixedStartupCost(i) & \text{if $L > 0$,}\\
		0 & \text{else.}
	\end{cases}}
\end{equation*}
In \eqref{equation:SC_FunctionForPeriods}, this function is discretized such that $\StartupCost(i,t,l)$ corresponds to the start-up cost incurred by unit~$i$ in period~$t$ after an offline time of $l$ periods. While the formulations 1-Bin, 1-Bin* and 3-Bin generally model a piece-wise constant approximation of these start-up costs, we use the unapproximated values $\StartupCost(i,t,l)$ for comparability with Temp.

\begin{samepage}
Each of the start-up cost formulations models $\startupCostSum$ differently:
\begin{itemize}
\item In 1-Bin (\cite{carrion_computationally_2006}), the start-up costs~$\startupCost(i,t)$ in period~$t$ are modeled as
\begin{alignat}{2}
	\label{equation:Ex_1Bin}
	\startupCost(i,t) &\geq \StartupCost(i,t,l) \Big( \onOff(i,t) - \sum_{\idx \in \discrange{1}{l}} \onOff(i,t-\idx) \Big) &\qquad&
	\forallUnitsPeriods, l \in \discrange*{0}{t-1},
\end{alignat}
and the summed start-up cost of unit~$i$ is $\startupCostSum := \sum_{t \in \Periods} \startupCost(i,t)$.

\item Similarly, 1-Bin* (\cite{silbernagl_improving_2014}) models $\startupCostSum := \sum_{t \in \Periods} \startupCost(i,t)$, with
\begin{alignat}{2}
	\label{equation:Ex_1BinT}
	\hskip-1em\startupCost(i,t) &\geq \StartupCost(i,t,l)\onOff(i,t) - \sum_{\mathclap{\idx \in \discrange{1}{l}}} \Big(\StartupCost(i,t,l) - \StartupCost(i,t,\idx\!-\!1)\Big) \onOff(i,t-\idx) &\quad& \forallUnitsPeriods, l \in \discrange*{0}{t-1}.%
\end{alignat}

\item 3-Bin (\cite{simoglou_optimal_2010}) uses the start-up and shutdown indicators (\cite{garver_power_1962}),
\begin{alignat}{2}
	\label{equation:Ex_IndicatorsFirst}
	\startupIndicator(i,1) - \shutdownIndicator(i,1) &= \begin{cases}
		\onOff(i,1) & \text{if $\preOffline(i) > 0$,}\\
		\onOff(i,1) - 1 & \text{else,}
	\end{cases} &\qquad& \forallUnits,\\
	\label{equation:Ex_Indicators}
	\startupIndicator(i,t) - \shutdownIndicator(i,t) &= \onOff(i,t) - \onOff(i,t-1) && \forallUnits, t \in \discrange{2}{T},\\
	\label{equation:Ex_IndicatorsNN}
	\startupIndicator(i,t),\shutdownIndicator(i,t) &\geq 0 && \forallUnitsPeriods.
\end{alignat}
In the special case of unapproximated start-up costs, this formulation simplifies to
\begin{alignat}{2}
	\startupIndicator(i,t) &= \sum_{l \in \discrange*{1}{t-1}} \startupType(i,t,l) &\qquad& \forallUnitsPeriods,\\
	\startupType(i,t,l) &\leq \shutdownIndicator(i,t-l) && \forallUnitsPeriods,l \in \discrange*{1}{t-2},\\
	\startupType(i,t,l) &\geq 0 && \forallUnitsPeriods,l \in \discrange*{1}{t-1},\\
	\startupCostSum &:= \sum_{t \in \Periods} \sum_{l \in \discrange*{1}{t-1}} \StartupCost(i,t,l) \startupType(i,t,l) && \forallUnits.
\end{alignat}

\item Temp (\cite{silbernagl_improving_2014}) also uses the start-up and shutdown indicators in \eqref{equation:Ex_IndicatorsFirst}-\eqref{equation:Ex_IndicatorsNN}, and models the start-up costs as
\begin{alignat}{2}
	\temp(i,1) &= \exp{-\heatloss(i)\preOffline(i)} + \heating(i,0) &\qquad& \forallUnits,\\
	\temp(i,t+1) &= \exp{-\heatloss(i)\periodLength(t)} \temp(i,t) + (1 - \exp{-\heatloss(i)\periodLength(t)}) \onOff(i,t) + \heating(i,t) && \forallUnits, t \in \discrange*{1}{T-1},\\
	\onOff(i,t) &\leq \temp(i,t) \leq 1  && \forallUnitsPeriods,\\
	\label{equation:EX_RTI0}
	\temp(i,t) &\geq \onOff(i,t) + \exp{-\heatloss(i)\offlineLength(t,t-1)}(1-\onOff(i,t)) && \forallUnitsPeriods,\\
	\label{equation:EX_RTI}
	\temp(i,t) &\geq \onOff(i,t) + \exp{-\heatloss(i)\offlineLength(t,l)}(\temp(i,t-l)-\onOff(i,t)) && \forallUnitsPeriods,l \in \discrange*{1}{t-2},\\
	\heating(i,t) &\geq 0 && \forallUnits,t \in \discrange*{0}{T-1},\\
	\startupCostSum &:= \fixedStartupCost(i) \sum_{t \in \Periods} \startupIndicator(i,t) + \heatingCost(i) \sum_{t \in \discrange*{0}{T-1}} \heating(i,t) && \forallUnits.
\end{alignat}
\item Finally BTI denotes the \hrepresentation of $\scsEpigraph$ in \cref{result:SCS_HRepresentation}.
\end{itemize}
\end{samepage}

\subsection{Integrality Gap}
\label{section:EX_IntGap}

\cref{figure:IntGap} shows the integrality gaps of all formulations for 14 test cases with $72$ one-hourly periods, \ie $T = 72$ and $\periodLength(1) = \ldots = \periodLength(T) = 1 \text{ hour}$. The modeled time ranges are spread uniformly over  the year 2025, starting in its $S$-th hour with $S \in \set{433 + 624\idx \given \idx \in \discrange{0}{13}}$. The values of $S$ are chosen such that each time range starts at midnight, and exactly two time ranges start on each day of the week.

\cite{silbernagl_thesis} shows that Temp is an extended formulation of $\scsEpigraph$, reducing the number of constraints by introducing the additional variables $\temp(i,t)$, $\heating(i,t-1)$ and $\startupIndicator(i,t)$ for all $i \in \Units$, $t \in \Periods$. Hence its integrality gap is equal to that of BTI, up to numerical inaccuracies. Both lead to the smallest integrality gap, which is expected since all start-up cost formulations model a set which encloses $\scsEpigraph$; The integrality gaps of 3-Bin, 1-Bin* and 1-Bin are significantly larger on average (\cref{figure:IntGap}, left chart). By expressing their integrality gaps relative to BTI, we see that BTI and Temp dominate the other models. (\cref{figure:IntGap}, right chart).

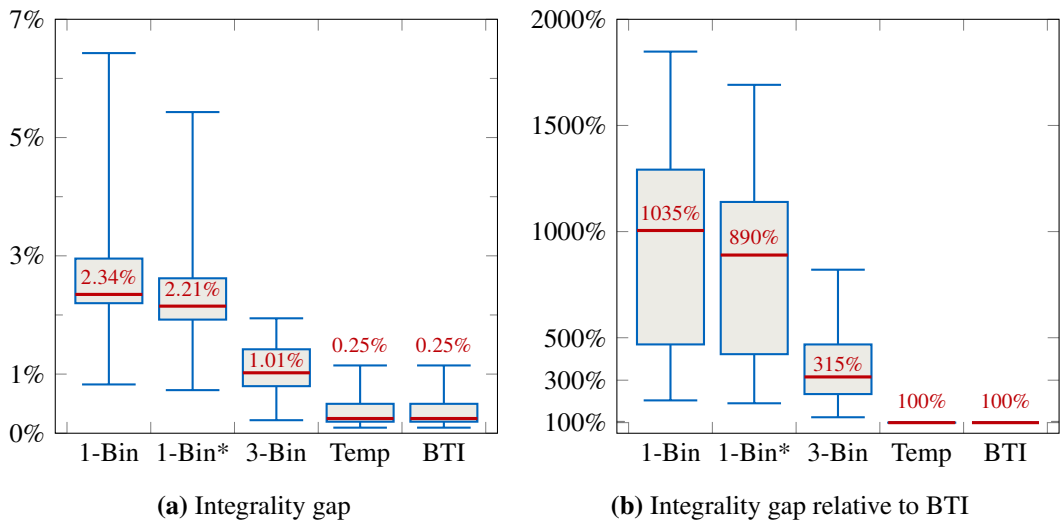
\begin{figure*}[htb]
	\label{fig:base}
	\newcommand{\chartwidth}{74mm}
	\newcommand{\chartheight}{70.7mm}
	\centering
	\begin{subfigure}[]{.477\columnwidth}
		\flushleft
		\begin{tikzpicture}
	\begin{axis}[
		tick label style={font=\small},	%
		boxplot/draw direction=y,
		width = \chartwidth,
		height = \chartheight,
		xtick={0.5,1.5,2.5,3.5,4.5,5.5},
		x tick label as interval=true,
		xticklabels={1-Bin,1-Bin*,3-Bin,Temp,BTI},
		enlarge x limits = 0.05,
		ymin=0,
		ymax=0.07,
		ytick={0,0.01,0.03,0.05,0.07},
		yticklabels={$0\%$,$1\%$,$3\%$,$5\%$,$7\%$},
		scaled y ticks = false,
		minor ytick={0.02,0.04,0.06},
		boxplot/every box/.style={fill=quantileBack,draw=quantileColor,mark=quantileColor,thick},
		boxplot/every whisker/.style={draw=quantileColor,thick},
		boxplot/every median/.style={draw=medianColor,very thick},
	]

		\addplot[mark=+,boxplot prepared={
			lower whisker=0.008249911,
			lower quartile=0.029534731,
			median=0.023473796,
			upper quartile=0.0219909,
			upper whisker=0.064282901
        }] table[row sep=\\,y index=0] {data\\};
        \node[anchor=south,medianColor] at (axis cs:1,0.023473796) {\scriptsize$2.34\%$};
		\addplot[mark=+,boxplot prepared={
			lower whisker=0.007288354,
			lower quartile=0.026200701,
			median=0.021492432,
			upper quartile=0.01921532,
			upper whisker=0.054309791
        }] table[row sep=\\,y index=0] {data\\};
        \node[anchor=south,medianColor] at (axis cs:2,0.021492432) {\scriptsize$2.21\%$};
		\addplot[mark=+,boxplot prepared={
			lower whisker=0.002203362,
			lower quartile=0.014194978,
			median=0.010220268,
			upper quartile=0.0079585,
			upper whisker=0.019434
        }] table[row sep=\\,y index=0] {data\\};
        \node[medianColor] at (axis cs:3,0.0123) {\scriptsize$1.01\%$};
		\addplot[mark=+,boxplot prepared={
			lower whisker=0.000943012,
			lower quartile=0.004972636,
			median=0.002488546,
			upper quartile=0.001950082,
			upper whisker=0.011463192
        }] table[row sep=\\,y index=0] {data\\};
        \node[medianColor] at (axis cs:4,0.015) {\scriptsize$0.25\%$};
		\addplot[mark=+,boxplot prepared={
			lower whisker=0.000943012,
			lower quartile=0.004972966,
			median=0.002488546,
			upper quartile=0.001951774,
			upper whisker=0.011463192
        }] table[row sep=\\,y index=0] {data\\};
        \node[medianColor] at (axis cs:5,0.015) {\scriptsize$0.25\%$};
    \end{axis}
\end{tikzpicture}%
		\caption{Integrality gap}
	\end{subfigure}
	\begin{subfigure}[]{.513\columnwidth}
		\centering
		\begin{tikzpicture}
	\begin{axis}[
		tick label style={font=\small},	%
		boxplot/draw direction=y,
		width = \chartwidth,
		height = \chartheight,
		xtick={0.5,1.5,2.5,3.5,4.5,5.5},
		x tick label as interval=true,
		xticklabels={1-Bin,1-Bin*,3-Bin,Temp,BTI},
		enlarge x limits = 0.05,
		ymin=0.5,
		ymax=20,
		ytick={1,3,5,10,15,20},
		yticklabels={$100\%$,$300\%$,$500\%$,$1000\%$,$1500\%$,$2000\%$},
		minor ytick={1.02,1.04,1.06},
		boxplot/every box/.style={fill=quantileBack,draw=quantileColor,mark=quantileColor,thick},
		boxplot/every whisker/.style={draw=quantileColor,thick},
		boxplot/every median/.style={draw=medianColor,very thick},
	]

		\addplot[mark=+,boxplot prepared={
			lower whisker=2.048432929,
			lower quartile=12.91824627,
			median=10.0512265,
			upper quartile=4.684969478,
			upper whisker=18.47902544
        }] table[row sep=\\,y index=0] {data\\};
        \node[anchor=south,medianColor] at (axis cs:1,10.0512265) {\scriptsize$1035\%$};
		\addplot[mark=+,boxplot prepared={
			lower whisker=1.911802364,
			lower quartile=11.39629007,
			median=8.895200798,
			upper quartile=4.225629963,
			upper whisker=16.9135323
        }] table[row sep=\\,y index=0] {data\\};
        \node[anchor=south,medianColor] at (axis cs:2,8.895200798) {\scriptsize$890\%$};
		\addplot[mark=+,boxplot prepared={
			lower whisker=1.253844115,
			lower quartile=4.683965487,
			median=3.153664099,
			upper quartile=2.345760198,
			upper whisker=8.204796746
        }] table[row sep=\\,y index=0] {data\\};
        \node[medianColor] at (axis cs:3,3.75) {\scriptsize$315\%$};
		\addplot[mark=+,boxplot prepared={
			lower whisker=0.991563807,
			lower quartile=0.99999999,
			median=0.999960426,
			upper quartile=0.999302728,
			upper whisker=1.000000023
        }] table[row sep=\\,y index=0] {data\\};
        \node[medianColor] at (axis cs:4,2) {\scriptsize$100\%$};
		\addplot[mark=+,boxplot prepared={
			lower whisker=1,
			lower quartile=1,
			median=1,
			upper quartile=1,
			upper whisker=1
        }] table[row sep=\\,y index=0] {data\\};
        \node[medianColor] at (axis cs:5,2) {\scriptsize$100\%$};
    \end{axis}
\end{tikzpicture}%
		\caption{Integrality gap relative to BTI}
	\end{subfigure}
	\caption{The integrality gap of the start-up cost formulation for 14 test cases, absolute (left chart) and relative to BTI (right chart). BTI and Temp provide equal integrality gaps, up to numerical inaccuracies, and dominate the other models.}
	\label{figure:IntGap}
\end{figure*}

Despite these results, augmenting 1-Bin, 1-Bin* and 3-Bin with binary tree inequalities as cutting planes does not improve the solution times considerably. While 1-Bin was able to solve more test cases within a time limit of 1 hour when paired with BTIs, no advantage was determined for 3-Bin. The cutting plane algorithm of the BTI exhibits a very slow convergence of the lower bound and numerical problems, leading to long computation times for each reoptimization step in the IBM ILOG CPLEX and the FICO Xpress Optimizers. Further research on the selection and implementation of the BTIs might overcome the numerical problems encountered.

Temp however does not suffer from the numerical problems and the slow convergence of the BTIs, and, as substantiated experimentally in \cite{silbernagl_improving_2014}, outperforms the start-up cost formulations in current literature. Still, the description of $\scsEpigraph$ we provide in this paper is the theoretical basis for Temp, and in particular justifies the separation algorithm for inequalities~\eqref{equation:EX_RTI0},\eqref{equation:EX_RTI}.

\section{Conclusion}
\label{section:Conclusion}

We were able to derive a complete \hrepresentation of the epigraph of the summed start-up costs. The resulting Unit Commitment formulation significantly outperforms two state-of-the-art formulations in terms of the integrality gap. While applying the BTIs in a cutting plane algorithm does not improve the solution times, they provide the theoretical basis for an efficient extended formulation.

In future work one may also take into account the interdependencies of multiple units for the development of further inequalities. Consider for example a demand-driven Unit Commitment problem: during a sharp increase in demand, some number of units may be forced to start up, resulting in a bound on the total start-up costs of all units, which has to be modeled by multi-unit inequalities not considered in any Unit Commitment model so far.

\bibliographystyle{apalike}
\bibliography{refs}

\end{document}